\documentclass[10pt,reqno]{amsart}

\usepackage[margin=1in]{geometry}        % or margin=0.9in for tighter copy
\usepackage{microtype}                   % better kerning & spacing

\usepackage{amsmath,amsthm,amssymb,amsfonts}
\usepackage{mathtools}

\usepackage[colorlinks=true, linkcolor=blue, citecolor=blue, urlcolor=blue]{hyperref}

\newtheorem{theorem}{Theorem}[section]
\newtheorem{lemma}[theorem]{Lemma}
\newtheorem{proposition}[theorem]{Proposition}

\theoremstyle{definition}
\newtheorem{definition}[theorem]{Definition}

\theoremstyle{remark}
\newtheorem{remark}[theorem]{Remark}

\newcommand{\R}{\mathbb{R}}
\newcommand{\D}{\mathfrak{D}}
\newcommand{\M}{\mathfrak{M}}
\newcommand{\loc}{\mathrm{loc}}
\newcommand{\Lz}{\operatorname{Sgn}}

\newcommand{\EH}{\mathrm{EH}}

\newcommand{\K}{\mathcal{K}}
\newcommand{\B}{\mathbb{B}}
\newcommand{\U}{\mathbb{U}}
\newcommand{\W}{\mathbb{W}}
\newcommand{\Q}{\mathrm{Q}}
\newcommand{\supp}{\operatorname{supp}}

\begin{document}

\title{A topology-changing variational framework for the Einstein-Hilbert functional}% Force line breaks with \\

\author{M. Paschalis}
\address{Department of Mathematics, National and Kapodistrian University of Athens, Greece}
\email{mpaschal@math.uoa.gr}

\date{\today}

\begin{abstract}
Motivated by recent developments in the theory of gravitation, we revisit the idea of topological variations, originally introduced by Wheeler and Hawking, from a rigorous perspective. Starting from a localized version of the Einstein-Hilbert variational principle, we encode the key aspects of the variational procedure in the form of a topology on a suitable space of Sobolev variational configurations, which is the final topology generated by the admissible variational maps. This framework naturally lends itself to generalization, and we rigorously introduce two distinct types of topological variations, corresponding to the infinitesimal addition of disconnected components and to infinitesimal surgeries, both motivated by related physical concepts. Using tools from the theory of Sobolev spaces and precise asymptotics, we establish dimensional obstructions for the continuity and differentiability of the Einstein-Hilbert action with respect to these variations, and show that in the extended variational framework the action does not admit critical points in dimension $n=4$, while higher dimensions are free of this problem. We also discuss the deeper geometric issue of scalar curvature blow-up of degenerating metrics within the context of our framework, and finally demonstrate the non-trivial effect of added higher order curvature terms on the critical dimension.
\end{abstract}


                              
\maketitle

\tableofcontents

\section{Introduction}

Variational principles have long constituted a universal framework for formulating the dynamics of physical systems. The calculus of variations has emerged as a separate, technically and conceptually rich field of mathematics, intimately related with the theory of partial differential equations and geometric analysis \cite{gelfand1963,landau1971,christodoulou2000}. In quantum field theory (QFT), variational formulations of physical theories are also of central importance to path integral approaches to quantization, as the action constitutes the primary input for the partition function; see \cite{simon1974,folland2008,dutsch2019} for rigorous treatments of path integrals. In the theory of general relativity (GR), the physical and geometric aspects of this framework are particularly intertwined. Einstein's insight was that given a background differentiable manifold $M$, the gravitational field is identified with a Lorentzian metric $g$ on $M$ obeying his celebrated field equations
\begin{equation}
R_{\mu\nu}-\frac{1}{2}Rg_{\mu\nu} + \Lambda g_{\mu\nu} = T_{\mu\nu},
\end{equation}
where the left hand side represents the Einstein tensor with an added cosmological constant term, and the right hand side represents the energy-momentum tensor which is determined by matter fields. Einstein \cite{einstein1915}, and almost simultaneously Hilbert \cite{hilbert1915}, demonstrated how the field equations can be derived from a variational principle, which is known today as the Einstein-Hilbert action and amounts to the integral of scalar curvature, a manifestly geometric quantity. Several equivalent or nearly equivalent variational principles for GR have been proposed through the years, which mostly amount to the addition of boundary terms or first order reformulations, aiming to either resolve problems regarding mathematical and physical aspects of the theory, or recast the theory in a form that can be generalized in different manners \cite{palatini1919,cartan1922,holst1996,york1972,gibbons1977,castrillon2014,parattu2015,lehner2016}.

The idea that background geometry is itself a dynamical variable has far-reaching implications and possible extensions. Wheeler \cite{wheeler1957} was the first to propose that since spacetime is identified with the Lorentzian manifold $(M,g)$, the whole structure is to be regarded as dynamic, not only the metric tensor $g$, meaning that the background topology and differentiable structure encoded in the differentiable manifold $M$ are also dynamical features of the theory. Taking into account the manifest scale-dependence of the metric tensor, he moreover conjectured that spacetime at the quantum level looks far less regular than it does at the macroscopic scale, and exhibits wild fluctuations in curvature and topology, an effect which Hawking \cite{hawking1978} later termed as \textit{spacetime foam}. Using a heuristic argument based on simplicial decompositions, Hawking suggested that Einstein-Hilbert action is continuous under changes of topology, in the sense that a change of topology that is suitably small in size changes the action by an arbitrarily small amount, in support to Wheeler's hypothesis. Hawking then proceeded to define a path integral in which the contributions of not only different geometries, but also topologies are summed. Nevertheless, in a later work \cite{hawking1982}, he argued that the loss of global hyperbolicity implicit in the process has as consequence the loss of quantum coherence, which sparked a debate on the viability of the theory by several authors including Coleman, Giddings and Strominger \cite{coleman1988,giddings1988}. It should be noted that even without topological variations, attempts to quantize the metric tensor have proven to be highly problematic due to catastrophic divergences \cite{thooft1974,goroff1986} or even cosmological loss of unitarity \cite{witten2001} from the QFT point of view. 

Regarding topological variations, Hawking's interpretation features an additional structural gap, namely the absence of a properly defined topological functional derivative. Without it, a variational interpretation of the theory is ill-posed, even in classical terms, because one cannot talk about critical points without a derivative. Moreover, both perturbative and non-perturbative approaches to the path integral are only meaningful if the stationary phase principle applies, which in turn requires a well-defined action admitting critical points, and thus well-defined functional derivatives. Without stationary phase, the probability measure does not concentrate about classical solutions as $\hbar \rightarrow 0$ and the quantum theory doesn't have a sensible classical limit. The difficulty in defining a topological functional derivative stems from the fact that there is no obvious infinitesimal version of a topological variation, in the sense that the metric tensor can be varied as $g\rightarrow g+\epsilon h$ and $\epsilon$ is taken to be arbitrarily small. To our knowledge, the first attempt to define a topological functional derivative fitted to the context of general relativity can be attributed to the recent work of Tsilioukas et al. \cite{tsilioukas2024}, in which the authors employ a heuristic semiclassical approach to investigate the effects of changing topology in Einstein-Gauss-Bonnet gravity. This work, however, rests on heuristic assumptions which - while physically motivated - cannot be justified from a rigorous point of view, and the formal introduction of a topological functional derivative remains an unresolved gap in the literature.

The long-standing gap together with the possibility of high-impact physical implications both warrant a more careful, rigorous analysis of topological variations. In this paper, we introduce a foundational approach aiming to establish a topological calculus of variations. In Section \ref{GR without topological variations}, we give a rigorous reformulation of the Einstein-Hilbert action as a localized variational principle, motivated by the local nature of Einstein's equations. We work with an appropriate space of admissible Sobolev metrics which offers a rich mathematical framework that allows to establish rigorous continuity results. The admissible variations are moreover encoded in the topology with which one equips the space of variational configurations, which in particular is the final topology with respect to the associated variational maps. In the case of purely geometric variations, this topology looks like a disjoint union of variational configurations over non-diffeomorphic structures, thus restricting the possibility of transitions from one background topology to another. 

In Section \ref{GR with topological variations} we argue that this topology can be refined to accommodate certain topological variations, while in Sections \ref{Disconnected topological variations} and \ref{Connected topological variations} we respectively introduce two types of infinitesimal topological variations, introducing new topology either via disconnected components or via surgery. We establish continuity of Einstein-Hilbert action with respect to the final topology of the associated deformation maps, which we further use to define topological functional derivatives. We compute these derivatives for the Einstein-Hilbert action, and establish the existence of a critical dimension which must be exceeded to ensure that the action admits critical points within the established framework, a feature absent from the regime of purely geometric variations. Interestingly, the critical dimension is equal to four, which implies that Einstein-Hilbert action has no critical points in that dimension when such topological variations are involved, which in turn brings forward issues with its well-posedness both in the classical and quantum setting.

In Section \ref{Blow-up}, we show that degenerating metrics can have diverging values for the Einstein-Hilbert functional while remaining uniformly bounded in their derivatives up to any order. Scalar curvature blow-up of degenerating metrics is a known phenomenon in geometric analysis (see, e.g. Nguyen \cite{nguyen2022}), but it appears that the specific realization we consider here is not explicitly discussed elsewhere. Moreover, we also discuss obstructions introduced by this phenomenon on the possibility of extending our proposed variational framework. Finally, in Section \ref{Quadratic terms}, we briefly showcase the non-trivial effect of critical dimension shift in the presence of higher order curvature terms.

Regarding the physical implications of our analysis, the conclusion is that topology-changing variational frameworks may in general impose nontrivial dimensional constraints on gravitational actions or even obstruct their continuity, and so caution is warranted when manipulating path integrals involving sums over multiple topologies as neither well-posedness of the path integrals nor proper classical limits are guaranteed. This means that, as a minimum, the sense in which topology change can occur has to be made precise and effects on the underlying action and the existence of critical points need to be investigated before making physical claims.

\section{The localized Einstein-Hilbert action} \label{GR without topological variations}

We begin our treatment with a rigorous mathematical formulation of classical GR without topological variations in terms of a local variational principle. While several variational principles have been proposed for GR, many differ from Einstein-Hilbert only by added boundary terms or use of the alternative tetrad framework, so without significant loss of generality we will focus here on Einstein-Hilbert. To avoid complicating the discussion, we will only consider the vacuum theory without matter fields and with no cosmological constant; these features can still be added later without problems. The starting point is the vacuum Einstein equations
\begin{equation}
R_{\mu\nu}-\frac{1}{2}Rg_{\mu\nu}=0, \label{Einstein equations}
\end{equation}
where $g_{\mu\nu}$ is the spacetime metric, $R_{\mu\nu}= {R^\rho}_{\mu\rho\nu}$ is the Ricci tensor obtained by contracting the curvature tensor of the Levi-Civita connection, and $R=g^{\mu\nu}R_{\mu\nu}$ is the scalar curvature, where $g^{\mu\nu}$ is the inverse of $g_{\mu\nu}$. In a classical context, these equations are assumed to hold on a background differentiable manifold $M$ of given, fixed topology and differentiable structure. It is also assumed that $\dim M = 4$ and that $g$ is of Lorentzian sinature $(-+++)$. In our discussion we generalize the setup by fixing the dimension $\dim M =n$ and signature $(r,s)$, where $r+s = n$; note that our conclusions will be independent of signature.

The Einstein equations \eqref{Einstein equations} are equivalently expressed as the vanishing of the variation of the localized Einstein-Hilbert action with respect to variations of the metric that have compact support. Given a precompact open set $\Omega\subset M$, let
\begin{equation}
S_\EH[M,g;\Omega]= \int_\Omega R[g] |\det g_{\mu\nu}|^{1/2}\,dx \label{Einstein-Hilbert}
\end{equation}
where $g$ is a metric of signature $(r,s)$ and $g_{\mu\nu}$ is its matrix representation in local coordinates $x^\mu$, $\mu=1,\ldots,n$. The Lagrangian density is an invariant quantity, independent of the coordinates in which it is expressed. Note that the action here is not a single functional but rather a family of functionals $\Omega \mapsto S_\EH[M,-;\Omega]$ where $\Omega$ ranges over the family of precompact regions in $M$. We avoid formulating the action in terms of a single global integral over $M$, as this would impose the global integrability of the Lagrangian density for the action to be well-defined. Since the equations \eqref{Einstein equations} are local in nature, this would be an \textit{ad hoc} assumption already yielding a non-equivalent interpretation of the theory by unjustifiably forcing this restriction. The approach presented here is in line with the general framework of local variational principles suggested by Christodoulou \cite{christodoulou2000}; also many others. In this approach, variations are performed with support inside an arbitrary precompact open subset of spacetime, rather than globally. This allows the space of admissible configurations to consist of fields that are merely locally integrable, avoiding unwanted restrictions.

The first task is to specify an appropriate domain of definition of the action, i.e. the space of admissible metrics $g$. In a formal functional analytic setup, it is often appropriate to work with Sobolev spaces. Their local version on the manifold $M$ can be defined in terms of an auxiliary smooth Riemannian metric $\hat{g}$ on $M$, in particular we declare a tensor field $T$ to belong to $W^{k,p}_\loc (M)$, $k\in\{0,1,\ldots\}$, $1\leq p < \infty$, if 
\begin{equation}
\| T \|_{W^{k,p}(K,\hat{g})} := \bigg\{ \sum_{j=1}^k \int_K |\hat{\nabla}^j T |_{\hat{g}}^p\,d\hat{v} \bigg\}^{1/p}
\end{equation}
is finite for every compact subset $K$ of $M$, where $|\cdot|_{\hat{g}}$, $\hat{\nabla}$ and $d\hat{v}$ denote the norm, weak covariant derivative and volume element defined by $\hat{g}$ respectively. The case $p=\infty$ is defined similarly with the $L^p$ norms replaced by the essential supremum; for Sobolev spaces in Riemannian manifolds, see Hebey \cite{hebey1996}. The local Sobolev norms $\| \cdot \|_{W^{k,p}(K,\hat{g})}$ do depend on the choice of $\hat{g}$, but they are equivalent for any compact set $K$ and so generate the same topology on $W^{k,p}_\loc (M)$ for tensors of any valence.

The expanded local expression of scalar curvature in terms of the inverse matrix $g^{\mu\nu}$ and the partial derivatives of $g_{\mu\nu}$ is
\begin{equation}\label{scalar curvature}
\begin{aligned}
R[g] &= \frac{1}{2} g^{\mu\nu} g^{\rho\sigma}
\big(
\partial_\rho \partial_\nu g_{\mu\sigma}
+ \partial_\sigma \partial_\mu g_{\nu\rho}
- \partial_\rho \partial_\sigma g_{\mu\nu}
- \partial_\mu \partial_\nu g_{\rho\sigma}
\big) \\
	&\hspace{1cm}+ \frac{1}{4} g^{\mu\nu} g^{\rho\sigma} g^{\alpha\beta}
\big(
(\partial_\rho g_{\nu\alpha})(\partial_\sigma g_{\mu\beta})
- 2 (\partial_\rho g_{\mu\alpha})(\partial_\sigma g_{\nu\beta})
\big),
\end{aligned}
\end{equation}
which is quasilinear in $g_{\mu\nu}$. From the form of \eqref{scalar curvature} we see that $g$ must be at least twice weakly differentiable, and also such that the Lagrangian density
\begin{equation}
\mathcal{L}[g] = R[g] |\det g_{\mu\nu}|^{1/2}\,dx
\end{equation}
is locally integrable for the action to be well defined. Note that due to the negative powers of $g_{\mu\nu}$ in \eqref{scalar curvature}, these requirements do not unambiguously place $g$ in a definite second order Sobolev space $W^{2,p}_\loc(M)$, as the later in general does not control the behavior of the inverse $g^{\mu\nu}$, which may be unbounded; the Lagrangian density thus may or may not be locally integrable for small values of $p$. If $p>n/2$, however, we may invoke the Morrey embedding $W^{2,p}_\loc(M) \hookrightarrow C^{0,\alpha}_\loc(M)$ (see, e.g., \cite{evans2010,adams2003,gilbarg2001}), i.e. those metrics are automatically $\alpha$-H\"older continuous. Since $g_{\mu\nu}$ as a matrix is invertible and continuous, it follows that $g_{\mu\nu}$, $g^{\mu\nu}$ and $|\det g_{\mu\nu}\,|$ are bounded above and bellow by positive constants inside any given compact set. Being an element of $W^{2,p}_{\loc}(M)$ automatically ensures local integrability of the $\partial_\sigma \partial_\rho g_{\mu\nu}$ in view of H\"older inequality, and for $p\geq 2n/(n+2)$ the Sobolev embedding $W^{2,p}_\loc(M) \hookrightarrow W^{1,2}_\loc(M)$ ensures that the $\partial_\rho g_{\mu\nu} \in L^2_\loc(M)$. Note that for $n\geq2$ we automatically have $n/2\geq 2n/(n+2)$, so we assume from now on $p>n/2$.

Let $\Lz M$ denote the bundle of all pointwise inner products of signature $(r,s)$ on $M$, i.e. all (0,2)-tensors $g_p:T_pM \times T_pM \rightarrow \R$ that can be brought to diagonal form $\mathrm{diag}(-1,\ldots,-1,1,\ldots,1)$ in some basis, where $-1$ appears $r$ times and $1$ appears $s$ times. The usual smooth metrics are then members of the space of smooth sections $C^\infty(\Lz M)$. In view of the preceding discussion, we define the domain of admissible metrics for the local Einstein-Hilbert action with background manifold $M$ to instead be
\begin{equation}
D_\EH(M) := W^{2,p}_\loc(\Lz M).
\end{equation}
When $M$ is compact, this space is an infinite-dimensional Banach manifold whose tangent space consists of all symmetric (0,2)-tensor fields of the same regularity; for the smooth case see Ebin \cite{ebin1970}. This space is independent of the domain $\Omega$ where the integration is carried out, and also doesn't feature any \textit{ad hoc} global restrictions, as promised. Note that, depending on signature, $D_\EH(M)$ may be empty for some $M$; for example, it is well known that while all differentiable manifolds can be equipped with a Riemannian metric, some do not admit a Lorentzian metric. Note that our analysis could be further simplified by instead considering the space $C^2(\Lz M)$ of twice continuously differentiable metrics, which also ensures local integrability of the Lagrangian density. However, this comes at the cost of losing the generality of weak derivatives, which we wish to retain for the sake of mathematical depth.

\begin{remark}
In several resources covering the variation of the Einstein-Hilbert action, metrics and their variations are assumed to be smooth. Working with smooth metrics is sufficient to almost recover the complete picture, since smooth functions are dense in Sobolev spaces. In the case of Riemannian signature, vacuum Einstein equations would be elliptic and solutions would automatically be smooth in view of elliptic regularity. Note, however, that if the signature is Lorentzian, solutions to the hyperbolic Einstein evolution equations are only as smooth as their initial data, which means there are perfectly eligible non-smooth solutions of the vacuum equations, so the a priori imposition of smoothness is another \textit{ad hoc} restriction; for standard references see \cite{evans2010,adams2003,gilbarg2001}.
\end{remark} 

So the action is, so far, a map $S_\EH: \{M \} \times D_\EH(M) \times \mathcal{K}(M) \rightarrow \R$ where $M$ is the given background manifold, $D_\EH(M)$ is the space of sufficiently regular metrics defined above, and $\K(M)$ is the set of all precompact open subsets of $M$. The Einstein equations \eqref{Einstein equations} are recovered by the usual variational procedure. Note, however, that the admissible variations do depend on the domain $\Omega$, as they need to have compact support in $\Omega$. While $D_\EH(M)$ determines the admissible metrics, the admissible variations are determined only once a domain $\Omega \in \K(M)$ is specified, and this information can be encoded in the topology with which the product $\{M \} \times D_\EH(M)\times \K(M)$ is equipped. It is appropriate to define this topology so that the desired variations are continuous by construction. For this we recall the notion of the final topology.

\begin{definition}
Let $\{(X_i,\tau_i)\}_{i\in I}$ be a collection of topological spaces, $Y$ a set, and $\mathcal{F} = \{ f_i : X_i \rightarrow Y \}_{i\in I}$ a family of maps. The \textit{final (or inductive limit) topology} $\tau(\mathcal{F})$ of $Y$ with respect to the family of maps $\mathcal{F}$ is defined by 
\begin{equation}
U\in \tau(\mathcal{F}) \text{ \textit{if and only if} for all } i\in I, \ f_i^{-1}(U)\in \tau_i.
\end{equation}
Note that $\tau(\mathcal{F})$ is the finest topology on $Y$ that makes all maps in $\mathcal{F}$ continuous.
\end{definition}

We denote by $W^{2,p}_c(\operatorname{Sym}\Omega)$ the space of all compactly supported symmetric $(0,2)$-tensor fields in $\Omega$, equipped with the subspace topology of $W^{2,p}_\loc(\operatorname{Sym}M)$. Since $\Omega$ is assumed to be precompact, this topology actually coincides with the subspace topology of the Banach space $W^{2,p}(\operatorname{Sym}\Omega)$. Note that $W^{2,p}_c(\operatorname{Sym}\Omega)$ is not a closed subspace of $W^{2,p}(\operatorname{Sym}\Omega)$, but this fact is of no consequence to our analysis.

\begin{remark}
There are several topology candidates for $W^{2,p}_c(\operatorname{Sym}\Omega)$, each having its appropriate use, so let us briefly justify the above choice as the most appropriate for our framework. First, it respects our choice of $W^{2,p}_\loc(\Lz M)$ as the appropriate space of admissible metrics, and indeed every construction that follows factors through $W^{2,p}_\loc(\Lz M)$ even when compactly supported variations are assumed; this is made apparent in Theorem \ref{Geometric continuity thm} and also arguments in subsequent sections that rely on it. It should also be noted that the compact support assumption is forced upon us for no reason other than avoiding discussing boundary terms, which has nothing to do with the intended regularity. Another possible option would be to consider the subspace topology in $W^{2,p}_\loc(\operatorname{Sym}\Omega)$, but this has the unwanted feature of uncontrolled boundary behavior and is ruled out for this reason.

The most common choice for a topology for $W^{2,p}_c(\operatorname{Sym}\Omega)$ in a distributional setting is arguably the final topology generated by the natural inclusion maps
\begin{equation}
\iota_K: W^{2,p}_K(\operatorname{Sym}\Omega) \hookrightarrow W^{2,p}_c(\operatorname{Sym}\Omega),
\end{equation}
where the $K$ range over all compact subsets of $\Omega$ and
\begin{equation}
W^{2,p}_K(\operatorname{Sym}\Omega):=\{ h\in W^{2,p}_c(\operatorname{Sym}\Omega): \supp(h) \subset K \}. 
\end{equation}
This is finer than the topology we identify as appropriate and finer than what is actually needed for our results. Moreover, since the Einstein-Hilbert Lagrangian density is a local differential operator and does not warrant a genuinely distributional approach; precise control of the supports is irrelevant here.
\end{remark}

The topology on $\{ M \} \times D_\EH(M) \times \K(M)$ is then defined to be the final topology with respect to variations of compact support.

\begin{definition}\label{inductive limit geometric variations}
\noindent
\begin{enumerate}
\item Let $g\in D_\EH(M)$, $\Omega \in \K(M)$, and let $\mathcal{U}(M,g;\Omega)$ be an open neighborhood of the zero section in $W^{2,p}_c(\operatorname{Sym}\Omega)$ such that the variational map $\Phi_{(M,g;\Omega)}: \mathcal{U}(M,g;\Omega) \rightarrow \{ M \} \times D_\EH(M) \times \K(M)$,
\begin{equation}
\Phi_{(M,g;\Omega)}(h):= (M,g+h;\Omega),
\end{equation}
is well defined, i.e. $g+h$ remains a metric of signature $(r,s)$ for all $h\in \mathcal{U}(M,g;\Omega)$.
\item We denote by $\tau_\EH(M)$ the final topology on $\{ M \} \times D_\EH(M) \times \K(M)$ with respect to the family of variational maps.
\end{enumerate}
\end{definition}

\noindent
Note that due to the Morrey embedding an open set $\mathcal{U}(M,g;\Omega)$ as in the definition always exists. This topology can be described explicitly by writing down a basis of open neighbourhoods. These are easily seen to be the open sets
\begin{align}\label{Open balls}
\begin{split}
B_\epsilon(M,g;\Omega,\hat{g}) := \{ (M,\tilde{g};\Omega): \tilde{g} \in D_\EH(M): \supp(\tilde{g}-g)  \subset \Omega,\ \| \tilde{g}-g \|_{W^{2,p}(\Omega,\hat{g})} <\epsilon \}.
\end{split}
\end{align}
This means that two metrics in $D_\EH(M)$ are close only if they agree outside a compact subset of the given domain of variation; in that case the closeness is measured in terms of the Sobolev norm, which is exactly what happens in the variational procedure. The crucial insight is that while $D_\EH(M)$ gives a description of admissible configurations, the topological space $\{ M\}\times D_\EH(M)\times\K(M)$ gives a prescription of admissible variations. Both aspects are important for the precise definition of the variational principle.

In this setup, the usual functional derivative may be introduced and gives rise to the Einstein tensor. Given $\Omega\in \K(M)$ and $g\in D_\EH(M)$, we consider a symmetric (0,2)-tensor field $h$ with compact support in $\Omega$. Then $\tilde{g}:=g+\epsilon h$ belongs to some open set of the form \eqref{Open balls} provided that $\epsilon$ is small enough, and we may define the functional derivative
\begin{align}
\begin{split}
\delta_{(M,g;\Omega)} S_\EH [h]:&= \frac{d}{d\epsilon}\bigg |_{\epsilon = 0} S_\EH[M,g+\epsilon h;\Omega] \\
&= \int_\Omega (R_{\mu\nu}-\frac{1}{2}Rg_{\mu\nu})h^{\mu\nu}|\det g_{\mu\nu}|^{1/2}\,dx
\end{split}
\end{align}
where the compact support assumption is used to eliminate boundary terms in a classic computation; for details see Wald \cite{wald1984}. Requiring this to be zero for all $h$ recovers the Einstein equations \eqref{Einstein equations} in view of the fundamental lemma of calculus of variations. The issue of boundary terms that occur via variation without the compact support assumption and the introduction of counter-terms that annihilate them is discussed by several authors, see for example \cite{york1972,gibbons1977,parattu2015,lehner2016}. As we are ultimately interested in topological variations, we deliberately work with compactly supported variations to avoid overcomplicating our analysis.

It is convinient, and at the same time instructive for our purposes, to generalise this picture in a way that doesn't assume a fixed background manifold $M$. Instead of regarding the action as a functional $S_\EH|_M:\{ M\} \times D_\EH(M) \times \mathcal{K}(M) \rightarrow \R$ for each $M$ separately, we may equivalently define a unifying functional $S_\EH: \D_\EH  \rightarrow \R$, where $\D_\EH$ is now a space of triples $(M,g;\Omega)$ where $M$ is a differentiable manifold, $g\in D_\EH(M)$ and $\Omega\in\K(M)$. The definition of the action reads the same, i.e.
\begin{equation}
S_\EH[M,g;\Omega] := \int_\Omega R[g] |\det g_{\mu\nu}|^{1/2}\,dx,
\end{equation}
but the fixed argument $M$ has been implicitly promoted to a variable. The structure of $\D_\EH$ in the case of purely geometric variations is also straightforward: it is a disjoint union, both-set theoretically and - most importantly - topologically, of the domains of each $S_\EH|_M$, i.e.
\begin{equation}\label{Domain of Einstein-Hilbert}
\D_\EH = \coprod_{ M\in\M } \{ M \} \times D_\EH(M)\times \K(M),
\end{equation}
where the disjoint union is understood to be over the space of all non-diffeomorphic differentiable manifolds $\M$, to imply that there can be no continuous transitions from one underlying topology and differentiable structure, encoded in $M$, to another. The functional derivatives in this case are also identical, since the disjoint union topology only permits changes in the metric tensor.

\begin{remark}
The assumption that $\M$ exists is justified from a set-theoretic point of view if one considers the Whitney embedding theorem \cite{whitney1936,hirsch1976,lee2013}; if $n$ is the dimension of manifolds and $N=N(n)$ is the Whitney dimension, we have a concrete identification $\M\subset \mathcal{P}(\R^N)$ and multiple copies of the same manifold embedded in different ways can be modded out by diffeomorphism.
\end{remark}

\begin{definition}
A triple $(M,g;\Omega)\in \D_\EH$ will be called a \textit{variational configuration}.
\end{definition}

While it might appear that not much has been gained by reformulating Einstein-Hilbert in this way, the reformulation comes with the benefit of encoding all the information regarding the variational procedure in the form of an appropriate topology on $\D_\EH$, which readily offers itself to generalizations. This topology determines which variational configurations are close, and hence which variations are admissible. The disjoint union topology rules out topological variations, but it can be argued that this restriction is somewhat artificial. We will show in the next sections that the topology on $\D_\EH$ can be refined in a way (actually multiple ways) that also permits topological variations and in which the action is still continuous and evenmoreso differentiable, in an appropriate sense.

\begin{remark}[Moduli space structure]
Two variational configurations $(M_1,g_1;\Omega_1)$ and $(M_2,g_2;\Omega_2)$ may be identified in a natural way if there exists an isometry $\phi:(M_1,g_1) \rightarrow (M_2,g_2)$ such that $\phi(\Omega_1) = \Omega_2$. The space of variational configurations $\D_\EH$ modulo this identification then constitutes a moduli space, where the moduli are equivalence classes of variational configurations. Our framework can thus be regarded as an instance of calculus of variations in moduli spaces.
\end{remark}

\begin{remark}[Associated diffeological structure]
Another thing to note is that the final topology which we introduce here (and also later), with respect to admissible variations, is naturally associated with a diffeological structure on $\D_\EH$. This concept naturally generalizes the notion of a differential structure, which the space of variational configurations does not admit in our setup, due to the fact that the model spaces $W^{2,p}_c(\operatorname{Sym} \Omega)$ are not isomorphic to each other as $\Omega$ ranges over possible precompact open subsets. Nevertheless, the variational maps do specify what the smooth plots are on $\D_\EH$, which is precisely what is required to define a diffeological structure. While we do not pursue this issue further in this generality, we will be concerned later with associated tangent structures that arise in our constructions.
\end{remark}

Bellow we prove the continuity of Einstein-Hilbert action with respect to geometric variations in our setup. This is important in its own right, but also very usefull for what follows.

\begin{theorem}\label{Geometric continuity thm}
Let $M$ be a fixed differentiable manifold of dimension $n\geq 2$ and let $p>n/2$. Then the Einstein-Hilbert Lagrangian density
\begin{equation}
\mathcal{L}:W^{2,p}_\loc(\Lz M) \rightarrow L^1_\loc(\Lambda^n M),\quad g\mapsto R[g]|\det g_{\mu\nu}|^{1/2}\,dx,
\end{equation}
is continuous with respect to the associated local seminorms of each space. In particular, the localized Einstein-Hilbert action $S_\EH: \D_\EH \rightarrow \R$ is continuous with respect to geometric variations.
\end{theorem}

\begin{proof}
Let $K\subset M$ be a compact subset, and observe that 
\begin{equation}\label{difference split}
\begin{aligned}
&|S_\EH(M,g;K)-S_\EH(M,g_*;K)| \\
	&\hspace{1cm}= \bigg| \int_K (R[g]|\det g_{\mu\nu}|^{1/2} - R[g_*]|\det g_{*\mu\nu}|^{1/2})\,dx \bigg| \\
	&\hspace{2cm}\leq \int_K |R[g]|\big| |\det g_{\mu\nu}|^{1/2} - |\det g_{*\mu\nu}|^{1/2} \big|\,dx + \int_K |R[g]-R[g_*] |\det g_{*\mu\nu}|^{1/2} \,dx.
\end{aligned}
\end{equation}
It suffices to show that each of the integrals in the last line of \eqref{difference split} converges to zero as $g\rightarrow g_*$ in $W^{2,p}_\loc(M)$.

Let $\Omega\in\mathcal{K}(M)$ be a precompact open set such that $K\subset \Omega$. Due to the embedding $W^{2,p}_\loc(M)\hookrightarrow C^{0,\alpha}_\loc(M)$, it follows that $\| g_{\mu\nu}-g_{*\mu\nu} \|_{L^\infty(K)} \leq C(\Omega,K,p) \| g-g_* \|_{W^{2,p}(\Omega)}$ for some positive constant $C(\Omega,K,p)$. Since $\det g_{\mu\nu}$ depends continuously on the $g_{\mu\nu}$, it follows that we also have $\| |\det g_{\mu\nu}|^{1/2} - |\det g_{*\mu\nu}|^{1/2} \|_{L^\infty(K)} \leq C(\Omega,K,p) \| g-g_* \|_{W^{2,p}(\Omega)}$ for another positive constant $C(\Omega,K,p)$. The same argument reveals that $\det g_{\mu\nu}$ is bounded away from zero in $K$ if $\| g-g_* \|_{W^{2,p}(\Omega)}$ is small, which means that we also have $\| g^{\mu\nu} - g_*^{\mu\nu} \|_{L^{\infty}(K)} \leq C(\Omega,K,p,g_*) \| g-g_* \|_{W^{2,p}(\Omega)}$ for yet another positive constant $C(\Omega,K,p,g_*)$. This proves that uniform convergence of all nonlinearities in $K$ is controlled by the norm of $W^{2,p}(\Omega)$. Under these conditions, we further have the estimate
\begin{equation}\label{g-1 estimates}
\| g^{\mu\nu} \|_{L^{\infty}(K)} \leq \| g_*^{\mu\nu} \|_{L^\infty(K)} +C(K,p,g_*) \| g-g_* \|_{W^{2,p}(\Omega)}.
\end{equation}
in view of the triangle inequality. Moreover, the embeddings $W^{2,p}_\loc(M)\hookrightarrow W^{2,1}_\loc(M)$ and $W^{2,p}_\loc(M)\hookrightarrow W^{1,2}_\loc(M)$ imply that there exists a positive constant $C(\Omega,K,p)$ such that 
\begin{equation}\label{sobolev estimates}
\| \partial_\sigma\partial_\rho g_{\mu\nu} - \partial_\sigma \partial_\rho g_{*\mu\nu} \|_{L^1(K)} + \| \partial_\rho g_{\mu\nu} - \partial_\rho g_{*\mu\nu} \|_{L^2(K)} \leq C(\Omega,K,p) \| g-g_* \|_{W^{2,p}(\Omega)},
\end{equation}
which also implies
\begin{equation}\label{sobolev estimates g}
\| \partial_\sigma \partial_\rho g_{\mu\nu} \|_{L^1(K)} + \| \partial_\rho g_{\mu\nu} \|_{L^2(K)} \leq \| \partial_\sigma \partial_\rho g_{*\mu\nu} \|_{L^1(K)} + \| \partial_\rho g_{*\mu\nu} \|_{L^2(K)} + C(\Omega,K,p)\| g-g_* \|_{W^{2,p}(\Omega)}.
\end{equation}

Recalling \eqref{scalar curvature}, setting $\delta g = g-g_*$ and combining estimates \eqref{g-1 estimates} and \eqref{sobolev estimates g}, we see that for $\| \delta g \|_{W^{2,p}(\Omega)}$ small enough there holds
\begin{align}\label{uniform bound}
\| R[g] \|_{L^1(K)} &\leq C(\Omega,K,p,g_*)(1+\| \delta g \|_{W^{2,p}(\Omega)}),
\end{align}
which is uniformly bounded as $\delta g\rightarrow 0$ in $W^{2,p}_\loc(M)$. Moreover, substituting $g=g_*+\delta g$ in \eqref{scalar curvature} and grouping linear terms in the difference $R[g_*+\delta g] - R[g_*]$, we obtain  
\begin{equation}
\begin{aligned}
R[g_*+\delta g] - R[g_*]
&= \frac{1}{2}
\Big(
\big((g_*+\delta g)^{\mu\nu}(g_*+\delta g)^{\rho\sigma}
      - g_*^{\mu\nu} g_*^{\rho\sigma}\big)
\big(
\partial_\rho \partial_\nu g_{*\mu\sigma}
+ \partial_\sigma \partial_\mu g_{*\nu\rho} \\
&\hspace{0.5cm}
- \partial_\rho \partial_\sigma g_{*\mu\nu}
- \partial_\mu \partial_\nu g_{*\rho\sigma}
\big) + (g_*+\delta g)^{\mu\nu}(g_*+\delta g)^{\rho\sigma} \big(
\partial_\rho \partial_\nu \delta g_{\mu\sigma} \\
&\hspace{1cm}
+ \partial_\sigma \partial_\mu \delta g_{\nu\rho} - \partial_\rho \partial_\sigma \delta g_{\mu\nu}
- \partial_\mu \partial_\nu \delta g_{\rho\sigma} \big)
\Big) \\
&\hspace{0.5cm}
+ \frac{1}{4}
\Big(
\big((g_*+\delta g)^{\mu\nu}(g_*+\delta g)^{\rho\sigma}(g_*+\delta g)^{\alpha\beta}
      - g_*^{\mu\nu} g_*^{\rho\sigma} g_*^{\alpha\beta}\big) \big(
\partial_\rho g_{*\nu\alpha}\,\partial_\sigma g_{*\mu\beta}
\\
&\hspace{1cm}
- 2\,\partial_\rho g_{*\mu\alpha}\,\partial_\sigma g_{*\nu\beta} \big) + (g_*+\delta g)^{\mu\nu}(g_*+\delta g)^{\rho\sigma}(g_*+\delta g)^{\alpha\beta}
\big(
\partial_\rho \delta g_{\nu\alpha}\,\partial_\sigma g_{*\mu\beta}
\\
&\hspace{1.5cm}
+ \partial_\rho g_{*\nu\alpha}\,\partial_\sigma \delta g_{\mu\beta} - 2\,\partial_\rho \delta g_{\mu\alpha}\,\partial_\sigma g_{*\nu\beta}
- 2\,\partial_\rho g_{*\mu\alpha}\,\partial_\sigma \delta g_{\nu\beta} \\
&\hspace{2cm}
+ \partial_\rho \delta g_{\nu\alpha}\,\partial_\sigma \delta g_{\mu\beta} - 2\,\partial_\rho \delta g_{\mu\alpha}\,\partial_\sigma \delta g_{\nu\beta}
\big)
\Big).
\end{aligned}
\end{equation}
In view of the estimates \eqref{g-1 estimates} for the inverse metric, we have that
\begin{equation}\label{g-2 estimates}
\begin{aligned}
&|(g_*+\delta g)^{\mu\nu}(g_*+\delta g)^{\rho\sigma} - g_*^{\mu\nu}g_*^{\rho\sigma}| \\
&\hspace{2cm}\leq |(g_*+\delta g)^{\mu\nu}| |(g_*+\delta g)^{\rho\sigma}-g_*^{\rho\sigma}| + |(g_*+\delta g)^{\rho\sigma}| |(g_*+\delta g)^{\mu\nu}-g_*^{\mu\nu}|  \\
&\hspace{4cm}\leq C(\Omega,K,p,g_*)(\| \delta g \|_{W^{2,p}(\Omega)} + \| \delta g \|_{W^{2,p}(\Omega)}^2)
\end{aligned}
\end{equation}
in $K$, and similarly
\begin{equation}\label{g-3 estimates}
\begin{aligned}
&|(g_*+\delta g)^{\mu\nu}(g_*+\delta g)^{\rho\sigma}(g_*+\delta g)^{\alpha\beta} - g_*^{\mu\nu}g_*^{\rho\sigma}g_*^{\alpha\beta}| \\
&\hspace{2cm} \leq C(\Omega,K,p,g_*)(\| \delta g \|_{W^{2,p}(\Omega)} + \| \delta g \|_{W^{2,p}(\Omega)}^2 + \| \delta g \|_{W^{2,p}(\Omega)}^3)
\end{aligned}
\end{equation}
in $K$. Combining the nonlinear estimates \eqref{g-2 estimates} and \eqref{g-3 estimates} with the linear estimate \eqref{sobolev estimates}, it follows that
\begin{align}\label{L1 difference}
\| R[g]-R[g_*] \|_{L^1(K)} &\leq C(\Omega,K,g_*,p) (\| \delta g \|_{W^{2,p}(\Omega)} + \| \delta g \|_{W^{2,p}(\Omega)}^2 +\| \delta g \|_{W^{2,p}(\Omega)}^3)
\end{align}

Returning to \eqref{difference split} and taking into account \eqref{uniform bound} and \eqref{L1 difference}, it follows that $|S_\EH(M,g;K)-S_\EH(M,g_*;K)| \rightarrow 0$ whenever $g\rightarrow g_*$ in $W^{2,p}_\loc(M)$. This completes the proof.
\end{proof}

\section{Conceptual foundations for topological variations} \label{GR with topological variations}

In the previous section, the variational procedure with respect to geometric variations was encoded in the topology of $\D_\EH$, which was conveniently a disjoint union of Sobolev topologies, so as to rule out topological variations. The later require a broader framework in which the underlying manifold may also vary, which we make precise in the present and following sections.

To start, we specify what it means for the manifold with metric $(M,g)$ to be varied topologically within a given precompact open subset $\Omega \subset M$. Since changes of topology in $M$ cannot be mediated by changes of the metric alone, a radical approach suggests itself. If the topology of $(M,g)$ is to change in $\Omega$, then the entire submanifold $(\Omega,g|_\Omega)$ (with boundary $\partial\Omega$) should be replaced with another manifold $(\tilde{\Omega},\tilde{g}|_{\tilde{\Omega}})$ (with boundary $\partial\tilde{\Omega}$). The word ``replaced'' is perhaps insufficient to capture the full implications of this process, as the new topological space which is to be formed by the union of $(M\setminus\Omega,g|_{M\setminus\Omega})$ and $(\tilde{\Omega},\tilde{g}|_{\tilde{\Omega}})$ should again form a manifold with metric, which is in general not guaranteed. The new topology $(\tilde{\Omega},\tilde{g}|_{\tilde{\Omega}})$ has to be glued in the place of the missing $(\Omega,g|_\Omega)$, which is possible only when certain junction conditions are satisfied. A comprehensive account on the gluing of spacetime manifolds is given in the recent book by Khakshournia and Mansouri \cite{khakshournia2023}. We will simply mention that the weakest junction condition which must hold for the manifolds to be glueable is that the boundaries should match, i.e. there should be an isometry $\phi: (\partial\tilde{\Omega},\tilde{\iota}^*\tilde{g}) \rightarrow (\partial\Omega,\iota^* g)$ between the boundaries with their respective induced hypersurface metrics, where $\iota,\tilde{\iota}$ are the inclusions of the boundaries in the respective manifolds with boundary; this is the classical Darmois (D) junction condition. In this case the gluing is a new manifold with metric $(\tilde{M},\tilde{g})$ which is given by
\begin{equation}\label{gluing}
\tilde{M} := (M\setminus\Omega)\sqcup\tilde{\Omega}/\phi, \quad \tilde{g}:= g|_{M\setminus\Omega}\oplus \tilde{g}|_{\tilde{\Omega}},
\end{equation}
where the quotient means that points of the boundaries are identified via the isometry, and the direct sum of metrics means that each set being glued is piecewise equipped with its original metric. This process is depicted in Figure \ref{Figure1}.

\begin{figure}[h]
\centering
\includegraphics[scale=0.175]{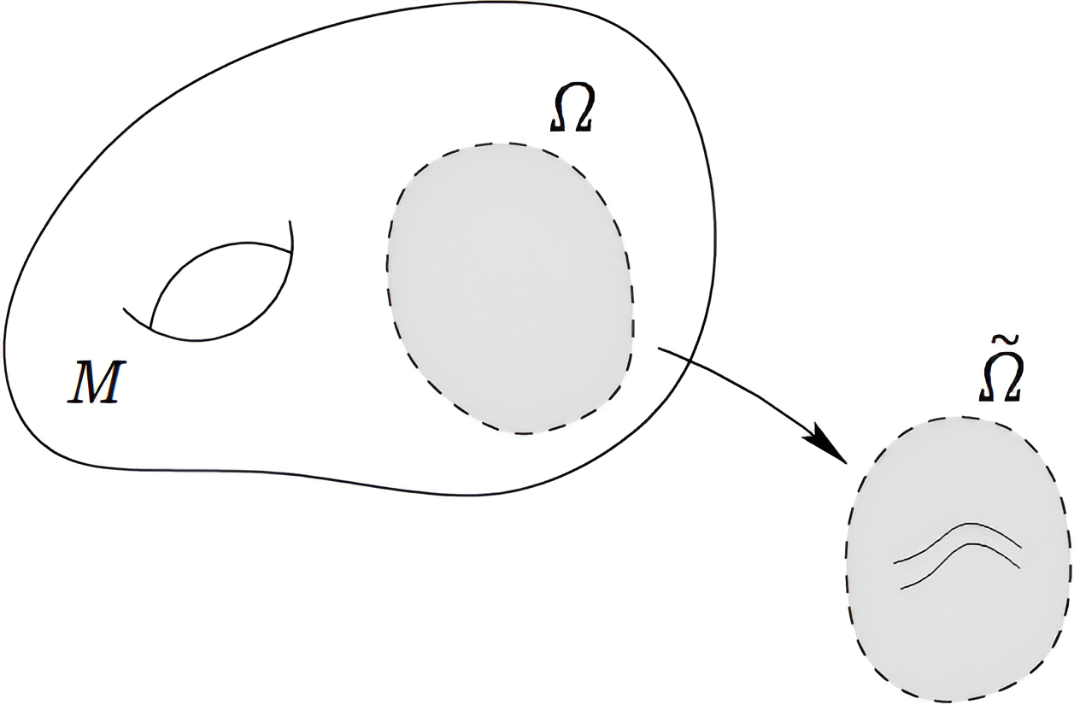}
\caption{Graphic impression of the topological variation. A domain of variation $\Omega$ is replaced with a new one $\tilde{\Omega}$ with new topology and metric, and matching boundary.}
\label{Figure1}
\end{figure}

Note that the (D) junction condition ensures that the metric $\tilde{g}$ is continuous in its components that are tangential to the gluing surface $\partial\Omega \equiv \partial\tilde{\Omega}$. This does not exclude the possibility of the metric being discontinuous in the normal direction; in particular we might have $\tilde{g}\not\in D_\EH(\tilde{M})$. This is problematic when one tries to define curvature, and for this reason stronger junction conditions are often employed, such as the Lichnerowicz (L) junction condition, which requires the normal components of the metric to be continuous as well; this allows for the realization of curvature as a distribution, see \cite{israel1966,mars2002,khakshournia2023}.

We will avoid the technical complications of junction conditions in a similar way that we avoided discussing boundary conditions in the previous section, by only considering variations of compact support. In the previous paragraph it was clear what compact support meant: the metric was unchanged in some open neighbourhood of the boundary. This notion can be extended to topological variations: we require that there exists an interior isometry between open neighborhoods $U$ and $\tilde{U}$ of the respective boundaries $\partial\Omega$ and $\partial\tilde{\Omega}$, say $\phi: (\tilde{U},\tilde{g}) \rightarrow (U,g)$. The recipe for the gluing is then the same as in \eqref{gluing}; this time any junction conditions are fulfilled automatically, since the gluing is the identical gluing in a neighborhood of the gluing hypersurface.

\begin{remark}
Since the initial domain of variation $\Omega$ is replaced with a new one $\tilde{\Omega}$, we see that in the proposed setup the domain of variation is itself subject to variation. The smooth variation of the domain with respect to a parametrisation of the boundary is known as the Weiss variation, and is associated with conserved quantities from the dynamical point of view. This is covered elsewhere (for instance see \cite{weiss1936,fengm2018}), and is different from what we do here; in the Weiss variation the topology of the underlying manifold and of the domain remains fixed.
\end{remark}

So a natural class of topological variations which matches the setup of Section \ref{GR without topological variations} consists of those obtained by compactly supported gluing operations. However, one runs into additional technical difficulties when attempting to define a topology for these variations. In the case of purely geometric variations, the basic open neighborhoods of $(M,g;\Omega)$ where families of metrics satisfying something like $\| \tilde{g} - g \|<\epsilon$ where the norm was the appropriate Sobolev norm for Einstein-Hilbert action, as seen in \eqref{Open balls}. In the case of topological variations, the metrics $\tilde{g}$ and $g$ belong to different tensor spaces (in particular, they are sections of bundles over different base manifolds $\tilde{M}$ and $M$) so it makes no sense to speak of their difference $\tilde{g}-g$, let alone a norm of that difference. A more subtle problem, equally important if we aspire to define derivatives, is the lack of an obvious differentiable or diffeological structure for this space, i.e. it is not clear what the infinitesimal topological variations should be.

Thus, a more sensible route is to instead adopt a bottom-up approach, introducing specific types of topological variations that can be made arbitrarily small in a suitable sense. While this means abandoning the complete generality of the framework presented above, it does indeed lead to a reasonable notion of infinitesimal topological variations. We will discuss two distinct types: disconnected and connected. The first type is the simplest to grasp conceptually and refers to new topology growing as disconnected components outside the original spacetime, also referred to as \emph{baby universes} in the terminology of Coleman, Giddings and Strominger \cite{coleman1988,giddings1988}. The second type refers to new topology arising as a connected sum on the original spacetime via controlled surgery, for example a small handle (wormhole) growing on the original spacetime, associated with \emph{spacetime foam} effects, as originally suggested by Wheeler and Hawking \cite{wheeler1957,hawking1978,hawking1982}.

\section{Disconnected topological variations}\label{Disconnected topological variations}

The simplest conceivable way of introducing new topology to a manifold is perhaps to add a new disconnected component. From the classical GR point of view, a disconnected component may be regarded as physically irrelevant, since it is causally inaccessible to observers on the initial spacetime. Following the intuition of Coleman, Giddings and Strominger \cite{coleman1988,giddings1988}, we instead argue that such an addition might be given the physical interpretation of the creation of a baby universe. Moreover, even disconnected components have non-trivial interactions from the quantum point of view, and hence are no longer physically irrelevant in that setting. These reasons suggest that topological variations in the form of disconnected components are physically motivated and should not be excluded from a reasonable formulation of quantum gravity and cosmology.

Mathematically, this can be implemented as follows. Recall that in Section \ref{GR without topological variations} variations where described not only in terms of an initial spacetime $(M,g)$, but rather a variational configuration $(M,g;\Omega)$ where $\Omega\subset M$ is a precompact open set in which the variation is supposed to take place. As mentioned in the previous section, in the regime of topological variations the domain itself is subject to change. Suppose now that we wish to add a closed disconnected component $(M',g')$ to the initial manifold. In this case the choice of $\Omega$ is not particularly important, as the variation is not really taking place ``inside'' any such $\Omega$. Mathematically, this addition amounts to the transformation
\begin{equation}\label{disjoint variation}
(M,g;\Omega) \longrightarrow (M\sqcup M', g \oplus g' ; \Omega \sqcup M').
\end{equation}
Nevertheless, the domain $\Omega$ is important when we consider mixed variations that vary both the metric within $\Omega$ and add the new component $(M',g')$ with whatever new topology and geometry it happens to carry.

Note that the topological variation \eqref{disjoint variation} is compatible with the general framework of topological variations outlined in Section \ref{GR with topological variations}. In particular, it coincides with the gluing \eqref{gluing} if we set 
\begin{equation}
(\tilde{\Omega}, \tilde{g}|_{\tilde{\Omega}}) = (\Omega \sqcup M', g|_\Omega \oplus g'|_{M'}),
\end{equation}
and is therefore a particular realization of this scheme. The disconnected topological variation is depicted in Figure \ref{Figure2}.

\begin{figure}[h]
\centering
\includegraphics[scale=0.225]{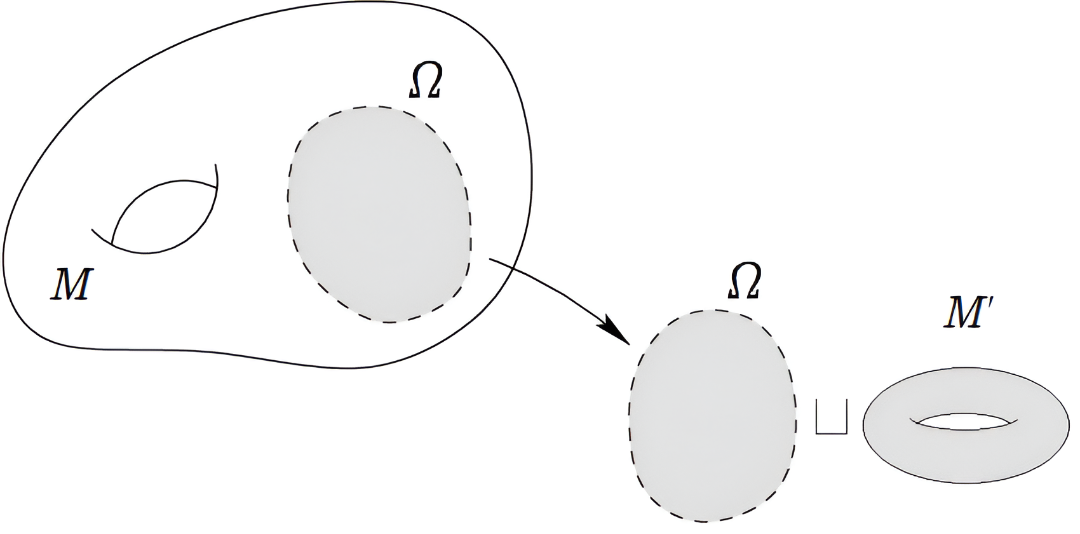}
\caption{Graphic impression of the disconnected topological variation. The topology of $\Omega$ remains unchanged, and a new disconnected component $M'$ is added. In the most general case, variation of the metric is permitted in $\Omega$, to reproduce the effects of classical geometric variations.}
\label{Figure2}
\end{figure}

To include disconnected topological variations in the case of Einstein-Hilbert action, let us first note that the domain of definition $\D_\EH$ as a set remains unchanged and still given by \eqref{Domain of Einstein-Hilbert}, i.e. the set-theoretic disjoint union
\begin{equation}\label{Domain of EH as set}
\bigcup_{M\in \M} \{M\} \times D_\EH(M) \times \K(M),
\end{equation}
consisting of variational configurations $(M,g;\Omega)$ as introduced in Section \ref{GR without topological variations}. The question now is what topology should this space be equipped with to include the notion of disconnected topological variations along with geometric ones. Motivated by the developments in Section \ref{GR without topological variations}, one may proceed to define a final topology with respect to both geometric and topological variations. 

There is, however, a subtle point that needs to be addressed, as it is not clear how the identity is recovered in variation \eqref{disjoint variation}. In fact, in the stated form the initial variational configuration $(M,g;\Omega)$ is not recovered for any $g' \in W^{2,p}(\Lz M')$, so technically this does not yet describe a variation of $(M,g;\Omega)$. The sense in which the identity is recovered is therefore dependent on the manner $g'$ exits the space $W^{2,p}(\Lz M')$, which hints towards degenerating configurations. The fact that the problem is motivated by the study of highly scale-dependent phenomena also suggests that the appropriate notion of proximity of the variation to the initial configuration must depend on the size of the adjoined manifold. A reasonable and physically motivated approach is therefore to consider limiting cases in which $g'$ degenerates; in this case $(M',g')$ is reduced to a lower dimensional object which has zero $n$-dimensional Hausdorff measure and can be ignored from this point of view. In fact, Hawking \cite{hawking1978} himself explicitly mentions both simplex length as a variational parameter and lower-dimensional collapse in his continuity argument based on simplicial decompositions. Topological variations of this type then occur via reverse-collapsing of lower dimensional objects.

The simplest realization of this idea is when the adjoined disconnected component degenerates to a point via a scaling of the metric. Mathematically this is further supported by the fact that $W^{2,p}(\Lz M')$ has the structure of an open \textit{cone} in the Banach space $W^{2,p}(\operatorname{Sym}M')$, i.e.
\begin{equation}
\forall g' \in W^{2,p}(\Lz M'):\ \forall \lambda>0:\ \lambda g' \in W^{2,p}(\Lz M'),
\end{equation}
so speaking of the limit $\lambda \rightarrow 0$ along any given ray is justified; in that case the Sobolev norm also converges to zero, in line with the regularity of our framework. These considerations also strongly suggest the following version of infinitesimal disconnected topological variation, given by
\begin{equation}\label{disjoint infinitesimal variation}
(M,g;\Omega) \longrightarrow (M\sqcup M', g \oplus \epsilon g' ; \Omega \sqcup M').
\end{equation}
where $\epsilon>0$ is taken to be arbitrarily small, and the identity corresponds to the limit $\epsilon=0$. The associated variational maps and final topology are made precise in the definition bellow.

\begin{definition}
\noindent
\begin{enumerate}
\item Given a variational configuration $(M,g;\Omega)$, let $\mathcal{U}(M,g;\Omega)$ be a neighborhood of the zero section in $W^{2,p}_c(\operatorname{Sym} \Omega)$ such that $g+h$ remains an admissible metric for all $h\in \mathcal{U}(M,g;\Omega)$. Assume, in addition, a closed manifold with metric $(M',g')$ with $g'\in D_\EH(M')$. We define the variational map $\Phi_{(M,g;\Omega, M', g')} : \mathcal{U}(M,g;\Omega)\times \R_+ \rightarrow \D_\EH$,
\begin{equation}
\Phi_{(M,g;\Omega, M', g')}(h,\epsilon) = \begin{dcases}
	(M,g+h;\Omega) & \text{if } \epsilon = 0 \\
  	(M\sqcup M', (g+h) \oplus \epsilon g' ; \Omega \sqcup M') & \text{if } \epsilon > 0
  \end{dcases}.
\end{equation}
\item Let $\tau_1$ denote the final topology of $\D_\EH$ with respect with the family of all variational maps as above.
\end{enumerate}
\end{definition}

\noindent
Note that this definition incorporates both geometric and disconnected topological variations in the same variational map; the case of purely geometric variations is recovered when $M'=\varnothing$.

\begin{remark}
Topology $\tau_1$ is the finest topology which promotes rays of $W^{2,p}(\Lz M')$ to directions of a tangent structure at the vertex $0\in W^{2,p}(\operatorname{Sym}M')$ of the cone $W^{2,p}(\Lz M')$. This tangent structure in particular coincides with the cone $W^{2,p}(\Lz M')$ itself with the origin adjoined. Since this paper is about the development of an infinitesimal topology-changing variational framework for geometric functionals arising in gravity, it will suffice for our purposes. Note however, that it is not the only appropriate one and certainly not the only one worth considering. The consideration of other natural candidates, including those permitting degeneration to limiting objects other than points, and also the discussion of associated geometric phenomena of scalar curvature blow-up which obstruct continuity of Einstein-Hilbert, will wait until Section \ref{Blow-up}. That discussion is enlightening regarding the context of choosing an appropriate topology and the challenges involved, and should help clarify the choices made here.
\end{remark}

Continuity of Einstein-Hilbert action under geometric variations has been established in Theorem \ref{Geometric continuity thm}. Based on heuristic arguments, Hawking \cite{hawking1978} has suggested that the action is continuous under at least certain topological variations. As we show in Section \ref{Blow-up}, the validity of this conclusion is actually highly dependent on the choice of variations, with many natural candidates making the action discontinuous. In what follows, we investigate continuity and differentiability of the Einstein-Hilbert functional in our proposed framework. First, we need the following lemma establishing the scaling properties of the Einstein-Hilbert Lagrangian density.

\begin{lemma}\label{Scaling Lemma}
Let $M$ be a differentiable manifold and let $g\in D_\EH(M)$. If $\mathcal{L}$ denotes the Einstein-Hilbert Lagrangian density on $M$, for any $\epsilon>0$ there holds
\begin{equation}
\mathcal{L}[\epsilon g] = \epsilon^{\frac{n-2}{2}}\mathcal{L}[g].
\end{equation}
\end{lemma}

\begin{proof}
Under the scaling $\epsilon \mapsto \epsilon g_{\mu\nu}$ we have $(\epsilon g)^{\mu\nu} = \epsilon^{-1} g^{\mu\nu}$, and the scalar curvature \eqref{scalar curvature} transforms as
\begin{equation}
R[\epsilon g] = \epsilon^{-1}R[g].
\end{equation}
Moreover, the volume element transforms as
\begin{equation}
|\det (\epsilon g_{\mu\nu})\,|^{1/2} = \epsilon^{n/2} |\det g_{\mu\nu}\, |^{1/2},
\end{equation}
and since $\mathcal{L}[g]=R[g] |\det g_{\mu\nu} |^{1/2} \,dx$, the conclusion easily follows.
\end{proof}

Now we prove the continuity of Einstein-Hilbert with respect to $\tau_1$.

\begin{theorem}\label{Disjoint continuity thm}
Let $(M,g;\Omega)\in \D_\EH$ be a given variational configuration, and let $(M',g')$ be a closed manifold with metric $g'\in D_\EH(M')$. Then the map 
\begin{equation}
S_\EH \circ \Phi_{(M,g;\Omega,M',g')} : \mathcal{U}(M,g;\Omega)\times \R_+ \rightarrow \R
\end{equation}
is continuous if and only if $n>2$ or $n=2$ and $g'$ has zero mean scalar curvature. In particular, $S_\EH: (\D_\EH,\tau_1) \rightarrow \R$ is continuous if and only if $n>2$.
\end{theorem}

\begin{proof}
Let $S(h,\epsilon):= S_\EH \circ \Phi_{(M,g;\Omega,M',g')}(h,\epsilon)$. We need to show that $S(h,\epsilon)\rightarrow S(h_*,\epsilon_*)$ whenever $\epsilon\rightarrow \epsilon_*$ in $\R_+$ and $h\rightarrow h_*$ in $W^{2,p}_c(\operatorname{Sym}\Omega)$. First observe that if $\epsilon_*>0$, then $\epsilon>0$ in a neighborhood of $\epsilon_*$ and
\begin{align}
S(h,\epsilon)-S(h_*,\epsilon_*) = \int_\Omega (\mathcal{L}[g+h] - \mathcal{L}[g+h_*]) + \int_{M'} (\mathcal{L}[\epsilon g']- \mathcal{L}[\epsilon_* g']).
\end{align}
Fixing $\epsilon=\epsilon_*>0$, the second term vanishes and $S(h,\epsilon_*)\rightarrow S(h_*,\epsilon_*)$ as $h\rightarrow h_*$ in $W^{2,p}_c(\operatorname{Sym}\Omega)$ in view of Theorem \ref{Geometric continuity thm}. Fixing $h=h_*$, the first term vanishes and Lemma \ref{Scaling Lemma} implies that
\begin{align}
S(h_*,\epsilon)-S(h_*,\epsilon_*) = \big(\epsilon^{\frac{n-2}{2}} - \epsilon_*^{\frac{n-2}{2}} \big) \int_{M'} \mathcal{L}[g'].
\end{align}
This is readily seen to converge to zero whenever $\epsilon_*>0$ and $\epsilon\rightarrow \epsilon_*$. 

The case $\epsilon_*=0$ is similar, although caution is required to manage the possibility of jumps and blow-ups. If $\epsilon=\epsilon_*=0$, the second term is again absent due to the fact that no disconnected component is added, and convergence in $h$ is established via Theorem \ref{Geometric continuity thm}. If $h=h_*$ is fixed and $\epsilon>0$, then Lemma \ref{Scaling Lemma} implies that
\begin{align}
S(h_*,\epsilon)-S(h_*,0) = \epsilon^{\frac{n-2}{2}} \int_{M'} \mathcal{L}[g'],
\end{align}
which converges to zero as $\epsilon\rightarrow 0^+$ if and only if $n>2$ or if $n=2$ and $g'$ has zero mean scalar curvature. This completes the proof.
\end{proof}

\noindent
This is the first instance in which we encounter a dimensional obstruction. Observe that if $n=2$, the action difference in the case of topological variation is $O(1)$ as the variation parameters approaches zero, which is a generically non-zero jump that renders the action discontinuous.

We are finally in a position to define a topological functional derivative for the case of disconnected topological variations.

\begin{definition}
\noindent
\begin{enumerate}
\item Let $(M,g;\Omega)\in \D_\EH$ be a fixed variational configuration, and let $(M',g')$ be a closed manifold with metric $g'\in D_\EH(M')$. Assume, moreover, that $F : (\D_\EH,\tau_1) \rightarrow \R$ is a continuous functional. The \textit{disconnected topological functional derivative of $F$ at $(M,g;\Omega)$ along $(M',g')$} is 
\begin{equation}
\delta_{(M,g;\Omega)} F [M',g'] := \frac{d}{d \epsilon} \bigg|_{\epsilon=0^+} F \circ \Phi_{(M,g;\Omega,M',g')}(0,\epsilon),
\end{equation}
provided that the one-sided derivative exists.
\item If the disconnected topological functional derivative exists for all $(M,g;\Omega)$ and all $(M',g')$, we say that $F$ is \textit{differentiable with respect to disconnected topological variations}.
\end{enumerate}
\end{definition}

\begin{remark}
It is worth noting that keeping the derivative one-sided is essential in order to maintain the signature of the variation; geometrically this is forced upon us by the tangent cone structure. Negative values of $\epsilon$ are easily seen to lead to reversed signature. 
\end{remark}

With a rigorous definition of a topological functional derivative in place, this discussion arrives at its culmination with the the following result regarding differentiability of Einstein-Hilbert within this framework.

\begin{theorem}\label{disjoint EH theorem}
Let $S_\EH: (\D_\EH,\tau_1) \rightarrow \R$ be the Einstein-Hilbert action extended to include disconnected topological variations of compact support, and let $n > 2$. Then the following statements are true.
\begin{enumerate}
\item If $n<4$, the action is not differentiable with respect to topological variations.
\item If $n=4$, the action is differentiable with respect to topological variations and the topological functional derivative is not identically zero. In particular, it is zero if and only if the variation has zero mean scalar curvature.
\item If $n>4$, the action is differentiable with respect to topological variations and the topological functional derivative is identically zero.
\end{enumerate}
\end{theorem}

\begin{proof}
Let $(M,g;\Omega)\in \D_\EH$ and $(M',g')$ a closed manifold with metric $g'\in D_\EH(M')$. We compute
\begin{equation}
S_\EH \circ \Phi_{(M,g;\Omega, M', g')}(0,\epsilon) = \begin{dcases}
	\int_\Omega \mathcal{L}[g] + \int_{M'} \mathcal{L}[\epsilon g'] & \text{if } \epsilon > 0 \\
  	\int_\Omega \mathcal{L}[g] & \text{if } \epsilon = 0
  \end{dcases}.
\end{equation}
It follows that for $\epsilon>0$ there holds
\begin{equation}\label{e-quotient disjoint final}
\frac{1}{\epsilon}\big[ S_\EH \circ \Phi_{(M,g;\Omega, M', g')}(0,\epsilon) - S_\EH \circ \Phi_{(M,g;\Omega, M', g')}(0,0) \big] = \epsilon^{\frac{n-4}{2}} \int_{M'} \mathcal{L}[g']. 
\end{equation}
The limit as $\epsilon \rightarrow 0^+$ exists if and only if $n\geq 4$, and is identically zero for all $(M',g')$ if and only if $n>4$. If $n=4$, it is zero if and only if the mean scalar curvature of $(M',g')$ is zero.
\end{proof}

Note that in $n=4$ the derivative is never vanishing along topological directions except for very specific ones. This means that the action does not have critical points in our framework in $n=4$, which is a rather unwholesome feature with potentially serious consequences. For $n<4$, the situation is even worse as a derivative doesn't exist at all, and one cannot talk about critical points to begin with. For $n>4$, however, the derivative is automatically zero in topological directions as seen from \eqref{e-quotient disjoint final} and the problem is resolved. In the latter case, Einstein manifolds are stationary points of the action with or without disconnected topological variations, and this leads to a consistent variational principle. 

\begin{remark}[The Kaluza-Klein model]
Our analysis indicates that, within this framework, the original variational interpretation of Einstein-Hilbert action in the extended regime of topological variations can be salvaged by introducing extra dimensions. It should be noted that the size or compactness of newly added dimensions is irrelevant to our results; the addition of even a tiny compact $S^1$ dimension is enough to resolve the issue. The simplest example of a model that works well under the topological variations introduced here is the otherwise problematic 5D Kaluza-Klein theory \cite{kaluza1921,klein1926,overduin1997}, which arises precisely in this manner.
\end{remark}

\begin{remark}[Stability of the zero metric]
In \cite{hawking1978}, Hawking has employed a method of path-integration over conformal classes and has suggested that the zero conformal factor of any metric constitutes a stationary point of the Einstein-Hilbert functional restricted to that conformal class. It can be argued, however, that fixing the conformal class is a strong constraint which drastically changes the problem. Our analysis shows that zero is definitely no longer a stationary point in $n=4$ once one abandons the convenient setting of conformal geometry.
\end{remark}

\section{Connected topological variations}\label{Connected topological variations}

The next step is to consider topological variations which attach new topology on a given manifold via surgery rather than as a disconnected component. This will be the case if, for example, we attach a small handle (or \textit{wormhole}, in physical terms) to the manifold and let its size vary. The implementation of this is going to be more complicated, as we need to take extra care for the gluing of the new topology and be mindful of how the gluing hypersurface changes as the variation scale changes. Regarding physical motivation, this type of variation corresponds directly to Wheeler and Hawking's original concept of spacetime foam.

The idea is roughly as follows. Scaling the metric by a factor of $\epsilon$ should scale lengths by a factor of $\sqrt{\epsilon}$. So we consider a manifold with boundary $(\tilde{B},\epsilon \tilde{g}|_{\tilde{B}})$ which has diameter scaling proportionally to $\sqrt{\epsilon}$ and will play the role of new growing topology. This will be glued in the place of a ball of diameter $\sqrt{\epsilon}$ of the initial manifold. Taking $\epsilon$ to be arbitrarily small achieves an infinitesimal version of this variation. Note that in the limit $\epsilon \rightarrow 0^+$ the new topology reduces to a point and for $\epsilon = 0$ we recover the original manifold. Given a domain of variation $\Omega \in \K(M)$ and a point $p\in \Omega$ around which the gluing is going to take place, $\epsilon$ can be chosen to be sufficiently small so that the gluing takes place within $\Omega$. This process is depicted in Figure \ref{Figure3}.

\begin{figure}[h]
\centering
\includegraphics[scale=0.75]{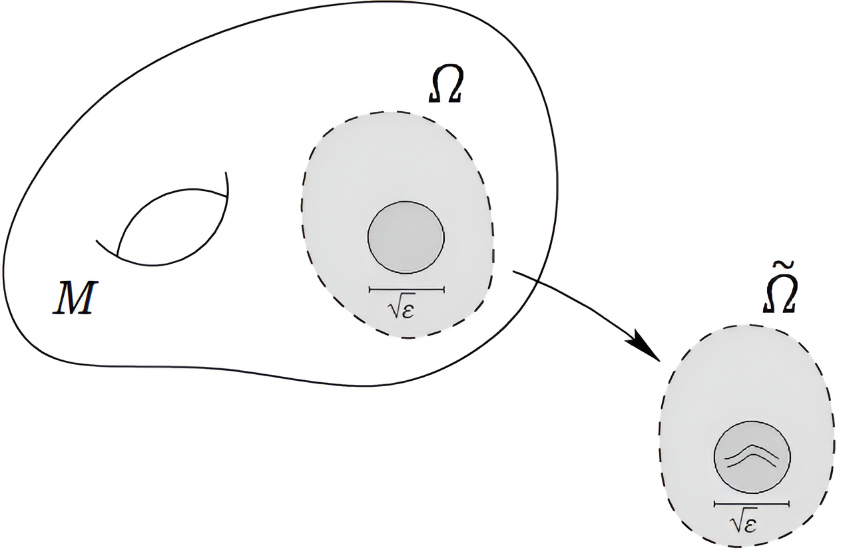}
\caption{Graphic impression of the connected topological variation. A ball of diameter $\sqrt{\epsilon}$ in $\Omega$ is replaced with another manifold with matching boundary, whose diameter also scales as $\sqrt{\epsilon}$. Since the metric scales by $\epsilon$, this is the topological equivalent of the geometric variation $g \rightarrow g+\epsilon h$.}
\label{Figure3}
\end{figure}

\hfill

\paragraph{\textbf{Connected topological variations in flat space.}} Due to the technical difficulties involved in ensuring gluing conditions along spheres of varying radii in a general background manifold $(M,g)$, and for instructive purposes, we will first implement the idea for the flat space $(\R^{n},\eta)$, where
\begin{equation}
\eta_{\mu\nu} = \mathrm{diag}(-1,\ldots,-1,1,\ldots,1)
\end{equation}
(this is Euclidean space for Euclidean signature $(r,s)=(0,n)$, and Minkowski space for Lorentzian signature $(r,s)=(1,n-1)$, and so on). The fully general case is developed in subsequent paragraphs using a rigorous surgery setup.

Let $\Omega\in \K(\R^{n})$ and $p\in \Omega$. Consider the Euclidean ball of radius $\sqrt{\epsilon}$ centered at $p$,
\begin{equation}
B_{\sqrt{\epsilon}}(p) := \{ x\in \R^{n} : (x^1)^2+ \cdots + (x^n)^2 < \epsilon \},
\end{equation}
where $(x^0,x^1,\ldots,x^{n-1})$ are normal coordinates spanned by an orthonormal basis at $p$. Since $\Omega$ is open, we always have $B_{\sqrt{\epsilon}}(p) \subset \Omega$ for $\epsilon$ small enough. The topological variation will be a scaling of a manifold with boundary $(\tilde{B}, \tilde{g}|_{\tilde{B}})$ with the following characteristics. First, the metric $\tilde{g}|_{\tilde{B}}$ should belong to the appropriate class of Sobolev metrics. Moreover, we want the boundaries of $\tilde{B}$ and $\B=B(0,1)\subset \R^n$ to match, and furthermore entire neighborhoods of those boundaries to match isometrically. This means there should be an isometry $\phi: (\tilde{U},\tilde{g}|_{\tilde{U}})\rightarrow (U,\eta|_U)$, where $\tilde{U}$ is a neighborhood of $\partial \tilde{B}$ and $U$ is a neighborhood of $\partial\B$. The symmetry of $\eta$ implies that $(B_{\sqrt{\epsilon}}(p),\eta)$ is isometric to $(B_1(p),\epsilon\eta)$. Thus $\phi$ can be promoted to an isometry between $(U,\epsilon\eta|_{U})$ and $(\tilde{U},\epsilon\tilde{g}|_{\tilde{U}})$, and the gluing 
\begin{align}\label{connected infinitesimal variation manifold}
\begin{split}
\tilde{M}&:= (\R^{n}\setminus B_{\sqrt{\epsilon}}(p)) \sqcup \tilde{B} / \phi, \\
\tilde{g}&:= \eta|_{\R^{n}\setminus B_{\sqrt{\epsilon}}(p)} \oplus \epsilon \tilde{g}|_{\tilde{B}},
\end{split}
\end{align}
where the quotient represents identification of points via the isometry, is well defined for any $\epsilon>0$ and any $p\in \R^{n}$. Note that this would \emph{not} be the case if $(\R^{n},\eta)$ was replaced with a generic manifold, as the balls of varying radii can no longer be expected to be isometric to the scaling of fixed ball.

The infinitesimal version of the connected topological variation in flat space is then
\begin{equation}\label{connected infinitesimal variation}
(\R^{n},\eta;\Omega) \longrightarrow (\tilde{M},\tilde{g}; \tilde{\Omega}), 
\end{equation} 
where $(\tilde{M},\tilde{g})$ is given by \eqref{connected infinitesimal variation manifold} and the domain $\Omega$, as a differentiable manifold, changes to
\begin{equation}
\tilde{\Omega}:= (\Omega\setminus B_{\sqrt{\epsilon}}(p)) \sqcup \tilde{B} / \phi.
\end{equation}

While we have not yet given a systematic treatment of the topology of connected topological variations and continuity of the action, it is still possible to calculate the difference in Einstein-Hilbert action brought forward by variation \eqref{connected infinitesimal variation}. This is
\begin{align}
\begin{split}
\Delta S(\epsilon) &= \int_{\Omega\setminus B_{\sqrt{\epsilon}}(p)} R[\eta] |\det \eta_{\mu\nu}|^{1/2}\,dx \\
	&\hspace{0.5cm}+ \int_{\tilde{B}} R[\epsilon \tilde{g}] |\det (\epsilon \tilde{g}_{\mu\nu})|^{1/2}\, dx \\
&\hspace{1cm} - \int_\Omega R[\eta] |\det \eta_{\mu\nu}|^{1/2}\,dx \\
	&= \epsilon^{\frac{n-2}{2}} \int_{\tilde{B}} R[\tilde{g}] |\det \tilde{g}_{\mu\nu}|^{1/2}\,dx,
\end{split}
\end{align}
where we have used the fact that $R[\eta] =0$ and Lemma \ref{Scaling Lemma}. We already observe exactly the same scaling behavior as in disconnected topological variations. A repetition of the same arguments reveals that the action is continuous under such variations provided that $n>2$ and differentiable provided that $n\geq 4$. For $n=4$, the functional derivative is
\begin{equation}
\lim_{\epsilon\rightarrow 0^+} \frac{\Delta S(\epsilon)}{\epsilon} =  \int_{\tilde{B}} R[\tilde{g}] |\det \tilde{g}_{\mu\nu}|^{1/2}\,dx,
\end{equation}
while for $n>4$ it is identically zero. This is generalized and fully justified in the next paragraphs.

\hfill

\paragraph{\textbf{The gluing setup.}} We begin here the treatment of connected topological variations in full generality. Let $(M,g;\Omega)\in \D_\EH$ be a given variational configuration. To obtain a notion of shrinking balls with smooth boundary as in the flat case, we will work in normal coordinates of some auxiliary smooth Riemannian metric $\hat{g}$. Let $p\in \Omega$ and consider an orthonormal basis $\{ \hat{e}_1,\ldots,\hat{e}_n \}$ of $T_p M$ with respect to $\hat{g}$, i.e.
\begin{equation}
\hat{g}_p(\hat{e}_\mu,\hat{e}_\nu) = \delta_{\mu\nu} = \operatorname{diag}(1,\ldots,1).
\end{equation}
Normal coordinates are then defined via the exponential mapping $(\hat{x}^1,\ldots,\hat{x}^n) \mapsto \exp^{\hat{g}}_p(\hat{x}^\mu \hat{e}_\mu)$, which is a diffeomorphism between an open neighboorhood of $0\in \R^n$ and an open neighborhood of $p\in M$. Since $\Omega$ is open, it follows that there exists $\hat{\epsilon}>0$, such that the coordinate balls
\begin{equation}
\hat{B}_\epsilon(p) = \{ q\in M: |\hat{x}(q)|<\epsilon \},\quad |\hat{x}|^2 = \sum_{i=1}^n (\hat{x}^i)^2,
\end{equation}
which coincide with the geodesic balls of $\hat{g}$ of radius $\epsilon$, are all contained in $\Omega$ for $0<\epsilon< \hat{\epsilon}$ and have the topology of a Euclidean ball.

Since we want the metric in the topological variation to scale proportional to $\epsilon$, the correct scaling for length is proportional to $\sqrt{\epsilon}$, and so the appropriate family of shrinking balls to consider is $\epsilon \mapsto \hat{B}_{\sqrt{\epsilon}}(p)$ for $0<\epsilon<\hat{\epsilon}^2$. Let us now remark how the general case is more complicated than the flat case explored earlier. Let 
\begin{equation}
\B = \{ \bar{x}\in \R^n : |\bar{x}|<1 \}
\end{equation}
be the unit Euclidean ball with its Cartesian coordinates $\bar{x} = (\bar{x}^1,\ldots, \bar{x}^n)$. The balls $(\hat{B}_{\sqrt{\epsilon}}(p),g)$ are in general \textit{not} isometric to a scaling of the form $(\B,\epsilon\bar{g})$ for some metric $\bar{g}$. This was exactly the case in the flat space $(\R^n,\eta)$ with $\hat{g}_{\mu\nu}=\delta_{\mu\nu}$ and $\bar{g}_{\mu\nu}=\eta_{\mu\nu}$, and allowed for the consistent gluing of a scaled manifold with boundary provided that it was flat itself near its boundary, without extra effort. To illustrate this point, consider the following counterexample: suppose we have a manifold which is everywhere flat except a small region, and let $p$ be a point in the vicinity of this region, but not inside it. Then for $\epsilon$ small enough, $\hat{B}_{\sqrt{\epsilon}}(p)$ does not intersect the non-flat region and is isometric to $(\B,\epsilon \eta)$, where $\eta$ here stands for the corresponding flat metric of $\R^{n}$. But once $\epsilon$ becomes large enough to intersect the non-flat region, the ball is no longer isometric to the above scaling.

So, on one hand, we are left to work with the family of balls $(\hat{B}_{\sqrt{\epsilon}}(p),g)$ as such, and ensure gluing conditions are met by employing more complicated methods than scaling. On the other hand we would still want some sort of scaling to take place, as this is our primary method for encoding the shrinking of the new topology to a point. Thus we have to provide a middle way that reconciles these two features: scaling of the new topology on one hand, and meeting of gluing conditions not obtainable by scaling on the other hand. Despite the fact that the balls are different manifolds with boundary, we would nevertheless like to have a unified description of their geometry in a manner that makes their metrics comparable. There is a way to do this, and starts with the observation that, as differentiable manifolds, the balls $\hat{B}_{\sqrt{\epsilon}}(p)$ are all diffeomorphic to $\B$. The diffeomorphisms
\begin{equation}
\phi_\epsilon: \B \rightarrow \hat{B}_{\sqrt{\epsilon}}(p),\quad \hat{x}^\mu\circ\phi_\epsilon(\bar{x}) = \sqrt{\epsilon}\cdot \bar{x}^\mu,
\end{equation}
obtained by the scaling of normal coordinates by $\sqrt{\epsilon}$, can be further promoted to isometries with respect to the original metric $g$ by pulling back the metric, i.e.
\begin{equation}
\phi_\epsilon: (\B, \bar{g}_\epsilon):= (\B, \phi_\epsilon^* g) \rightarrow (\hat{B}_{\sqrt{\epsilon}}(p) , g).
\end{equation}
In this way, all the information about the geometry of the shrinking balls is now contained in the family $(\B,\bar{g}_\epsilon)$. A useful property of this setup is that the deformation $\epsilon \mapsto \bar{g}_\epsilon$ is continuous in the appropriate Sobolev norm.

\begin{lemma}\label{Contraction Lemma}
Let $g\in W^{2,p}(\Lz \Omega)$. Then the map $\epsilon \mapsto \bar{g}_\epsilon := \phi_\epsilon^*g$ is continuous from $(0,\hat{\epsilon})$ to $W^{2,p}(\Lz \B)$.
\end{lemma}

\begin{proof}
First observe that, by the definition of the pullback, there holds
\begin{equation}\label{scaling metric}
\begin{aligned}
(\phi_\epsilon^*g)_{\bar{\mu}\bar{\nu}}(\bar{x})&= (\phi_\epsilon^*g) \bigg( \frac{\partial}{\partial \bar{x}^\mu}\bigg|_{\bar{x}},\frac{\partial}{\partial \bar{x}^\nu}\bigg|_{\bar{x}} \bigg) \\
	&= g \bigg( (\phi_{\epsilon})_* \frac{\partial}{\partial \bar{x}^\mu}\bigg|_{\bar{x}},(\phi_{\epsilon})_*  \frac{\partial}{\partial \bar{x}^\nu}\bigg|_{\bar{x}} \bigg) \\
	&=g \bigg( \sqrt{\epsilon}\frac{\partial}{\partial \hat{x}^\mu}\bigg|_{\phi_\epsilon(\bar{x})},\sqrt{\epsilon}\frac{\partial}{\partial \hat{x}^\nu}\bigg|_{\phi_\epsilon(\bar{x})} \bigg) \\
	&=\epsilon g_{\hat{\mu}\hat{\nu}}\circ\phi_\epsilon(\bar{x}).
\end{aligned}
\end{equation}
The weak derivatives of $\phi_\epsilon^*g$ up to order two are then
\begin{equation}
(\phi_\epsilon^*g)_{\bar{\mu}\bar{\nu},\bar{\rho}}(\bar{x}) = \epsilon^{3/2} g_{\hat{\mu}\hat{\nu},\hat{\rho}}\circ \phi_\epsilon(\bar{x}), \quad (\phi_\epsilon^*g)_{\bar{\mu}\bar{\nu},\bar{\rho}\bar{\sigma}}(\bar{x}) = \epsilon^{2} g_{\hat{\mu}\hat{\nu},\hat{\rho}\hat{\sigma}}\circ \phi_\epsilon(\bar{x}).
\end{equation}
So if $\beta$ denotes a multi-index of order $|\beta|\leq 2$, we have shown that
\begin{equation}
\partial_{\bar{\beta}}(\phi_\epsilon^* g)_{\bar{\mu}\bar{\nu}}(\bar{x}) = \epsilon^{1+|\beta|/2} \partial_{\hat{\beta}}g_{\hat{\mu}\hat{\nu}}\circ \phi_\epsilon(\bar{x}).
\end{equation}
It follows that for $0<\epsilon_*<\hat{\epsilon}$,
\begin{equation}
\begin{aligned}
\| \phi_\epsilon^*g - \phi_{\epsilon_*}^*g \|_{W^{2,p}(\B,\bar{g})} &\leq C \sum_{|\beta|\leq 2} \sum_{\mu,\nu=1}^{n} \| \partial_{\bar{\beta}} (\phi_\epsilon^*g)_{\bar{\mu}\bar{\nu}} - \partial_{\bar{\beta}} (\phi_{\epsilon_*}^*g)_{\bar{\mu}\bar{\nu}} \|_{L^p(\B,\bar{g})} \\
	&= C \sum_{|\beta|\leq 2} \sum_{\mu,\nu=1}^{n} \| \epsilon^{1+|\beta|/2} \partial_{\hat{\beta}} g_{\hat{\mu}\hat{\nu}}\circ \phi_\epsilon - \epsilon^{1+|\beta|/2} \partial_{\hat{\beta}} g_{\hat{\mu}\hat{\nu}}\circ \phi_{\epsilon_*} \\
	&\hspace{0.5cm}+ \epsilon^{1+|\beta|/2} \partial_{\hat{\beta}} g_{\hat{\mu}\hat{\nu}}\circ\phi_{\epsilon_*} - \epsilon_*^{1+|\beta|/2} \partial_{\hat{\beta}} g_{\hat{\mu}\hat{\nu}}\circ\phi_{\epsilon_*} \|_{L^p(\B,\bar{g})} \\
	&\leq C \sum_{|\beta|\leq 2} \sum_{\mu,\nu=1}^{n} \big\{ \epsilon^{1+|\beta|/2} \| \partial_{\hat{\beta}} g_{\hat{\mu}\hat{\nu}}\circ\phi_\epsilon - \partial_{\hat{\beta}} g_{\hat{\mu}\hat{\nu}}\circ\phi_{\epsilon_*} \|_{L^p(\B,\bar{g})} \\
	&\hspace{0.5cm}+ |\epsilon^{1+|\beta|/2}-\epsilon_*^{1+|\beta|/2}|\cdot \| \partial_{\hat{\beta}} g_{\hat{\mu}\hat{\nu}}\circ\phi_{\epsilon_*} \|_{L^p(\B,\bar{g})} \big\}
\end{aligned}
\end{equation}
For the second term in the last line we have that
\begin{equation}
\begin{aligned}
\| \partial_{\hat{\beta}} g_{\hat{\mu}\hat{\nu}} \circ \phi_{\epsilon_*} \|_{L^p(\B,\bar{g})} &= \| \partial_{\hat{\beta}} g_{\hat{\mu}\hat{\nu}} \|_{L^p(\hat{B}_{\sqrt{\epsilon_*}}(p),(\phi_{\epsilon_*})_* \bar{g})} \\
	&\leq C(\epsilon_*,\hat{g}) \| \partial_{\hat{\beta}} g_{\hat{\mu}\hat{\nu}} \|_{L^p(\hat{B}_{\sqrt{\epsilon_*}}(p),\hat{g})} \\
	&\leq C(\epsilon_*,\hat{g}) \| g \|_{W^{2,p}(\Omega,\hat{g})} \\
	&< \infty,
\end{aligned}
\end{equation}
where the diffeomorphism invariance of the integral and the equivalence of Sobolev norms in compact sets has been used in that order. Since $|\epsilon^{1+|\beta|/2}-\epsilon_*^{1+|\beta|/2}| \rightarrow 0$ as $\epsilon\rightarrow \epsilon_*$, this establishes convergence of the second term to zero. Regarding the first term, first note that $\epsilon^{1+|\beta|/2}$ is positive and bounded as $\epsilon\rightarrow \epsilon_*$. By virtue of uniform continuity of continuous functions in compact sets, it is straightforward to show that 
\begin{equation}
\| \partial_{\hat{\beta}} g_{\hat{\mu}\hat{\nu}}\circ\phi_\epsilon - \partial_{\hat{\beta}} g_{\hat{\mu}\hat{\nu}}\circ\phi_{\epsilon_*} \|_{L^p(\B,\bar{g})} \rightarrow 0
\end{equation}
as $\epsilon\rightarrow \epsilon_*$ when $\partial_{\hat{\beta}} g_{\hat{\mu}\hat{\nu}}$ is continuous. The general case where $\partial_{\hat{\beta}} g_{\hat{\mu}\hat{\nu}} \in L^p(\Omega)$ follows from a well-known density argument. This completes the proof.
\end{proof}

The topological variation is obtained by gluing a manifold with boundary along the boundary of the removed geodesic balls. The new topology itself should grow as a scaling $\epsilon \tilde{g}$ for some metric $\tilde{g}$, while near the boundary its metric should agree with the metric of the shrinking balls, or equivalently the metric $\bar{g}_\epsilon$. This is achieved as follows. Let $\tilde{B}$ be a compact manifold with boundary, whose boundary $\partial \tilde{B}$ is diffeomorphic to $\partial \B = S^{n-1}$. Then there exist tubular neighborhoods $\W=\B\setminus B(0,\epsilon_1)$ of $\partial \B$ and $\tilde{W}$ of $\partial \tilde{B}$ (also called \textit{collars}, see e.g. Lee \cite{lee2013}) and a diffeomorphism $\psi: \tilde{W} \rightarrow \W$. Note that since $\B$ and $\tilde{B}$ may have different topology, the diffeomorphism may not be extended over the entire manifolds; if that was the case we would essentially be back to the classical framework of geometric variations. By pulling back the metric $\bar{g}_\epsilon$, we can promote $\psi$ to a family of isometries 
\begin{equation}
\psi_\epsilon: (\tilde{W},\tilde{g}_\epsilon):= (\tilde{W}, \psi^* \bar{g}_\epsilon) \rightarrow (\W,\bar{g}_\epsilon).
\end{equation}
This ensures that gluing conditions are met between $(\tilde{W},\tilde{g}_\epsilon)$ and the complement of the shrinking ball $\hat{B}_{\sqrt{\epsilon}}(p)$ in $\Omega$.

To ensure the smooth transition between the metric $\epsilon\tilde{g}$ in the central region, where the new topology emerges, and the metric $\tilde{g}_\epsilon$ in the boundary region, we will employ a bump function argument. In particular, let $\U = \B\setminus B(0,\epsilon_2) \subsetneqq \W \subsetneqq \B$ be a neighborhood of $\partial\B$ and $\tilde{U} \subsetneqq \tilde{W} \subsetneqq \tilde{B}$ of $\partial \tilde{B}$ such that $\psi(\tilde{U}) = \U$. Take the metric to be equal to $\tilde{g}_\epsilon|_{\tilde{U}}$ on $\tilde{U}$ and equal to $\epsilon\tilde{g}|_{\tilde{B}\setminus\tilde{W}}$ in $\tilde{B}\setminus\tilde{W}$. What remains is to define the metric in the intermediate region $\tilde{W}\setminus\tilde{U}$. This is going to be the transition region. Let $f:\tilde{B} \rightarrow [0,1]$ be a smooth cutoff function such that $f|_{\tilde{B}\setminus\tilde{W}} = 1$ and $f|_{\tilde{U}} = 0$. The metric on $\tilde{W}\setminus\tilde{U}$ is then defined as the pointwise convex combination 
\begin{equation}
\tilde{g}_\epsilon = \epsilon f \tilde{g}|_{\tilde{W}\setminus\tilde{U}} + (1-f)\psi^* \bar{g}_\epsilon|_{\W\setminus\U}
\end{equation}
One should exercise caution, however, as convex combinations of metrics that have signature other than Riemannian may end up being degenerate although the metrics themselves are not. We therefore explicitly \textit{assume} that this is not the case, which places an additional requirement on $\tilde{g}$. This is not terribly restrictive however; by working in an orthonormal basis one can show that there are many metrics satisfying this requirement. In our setup, this matching property depends only on $g$ and $\tilde{g}$ and is not affected as $\epsilon$ runs small, as we prove bellow.

\begin{lemma}\label{linear interpolation lemma}
Let $\tilde{g}$ be such that the pointwise convex combination $f(\psi_*\tilde{g})_{\bar{\mu}\bar{\nu}}+(1-f)g_{\hat{\mu}\hat{\nu}}(p)$ is an element of $W^{2,p}(\Lz \W\setminus\U)$. Then 
\begin{equation}
\tilde{g}_\epsilon = \epsilon f \tilde{g} + (1-f)\psi^* \bar{g}_\epsilon
\end{equation}
is an element of $W^{2,p}(\Lz \tilde{W}\setminus\tilde{U})$ for all positive and suitably small $\epsilon$. In particular, 
\begin{equation}\label{transition approximate scaling}
(\psi_*\tilde{g}_\epsilon)_{\bar{\mu}\bar{\nu}}(\bar{x}) = \epsilon (f(\psi_* \tilde{g})_{\bar{\mu}\bar{\nu}}(\bar{x}) + (1-f)g_{\hat{\mu}\hat{\nu}}(p)) + O(\epsilon^{1+\alpha/2}),
\end{equation}
where $\alpha=\alpha(n,p)>0$ is the exponent of the Morrey embedding.
\end{lemma}

\begin{proof}
It is more convenient to work in $\W\setminus\U$. Note that due to \eqref{scaling metric}, we have that
\begin{equation}
\begin{aligned}
(\psi_*\tilde{g}_\epsilon)_{\bar{\mu}\bar{\nu}}(\bar{x}) &= \epsilon f (\psi_* \tilde{g})_{\bar{\mu}\bar{\nu}}(\bar{x}) + (1-f)(\phi_\epsilon^* g)_{\bar{\mu}\bar{\nu}}(\bar{x}) \\
	&= \epsilon (f(\psi_* \tilde{g})_{\bar{\mu}\bar{\nu}}(\bar{x}) +(1-f) g_{\hat{\mu}\hat{\nu}}\circ \phi_\epsilon(\bar{x})).
\end{aligned}
\end{equation}
In view of the embedding $W^{2,p}_\loc(M) \hookrightarrow C^{0,\alpha}_\loc(M)$, it follows that
\begin{equation}
|g_{\hat{\mu}\hat{\nu}}(q)- g_{\hat{\mu}\hat{\nu}}(p)| \leq C |\hat{x}(q)-\hat{x}(p)|^\alpha
\end{equation}
in $\Omega$. Setting $q=\phi_\epsilon(\bar{x})$ and noting that $p=\phi_\epsilon(0)$ and $\hat{x}^\mu \circ \phi_\epsilon(\bar{x}) = \epsilon^{1/2} \bar{x}^\mu$, we see that
\begin{equation}
|g_{\hat{\mu}\hat{\nu}}\circ\phi_\epsilon(\bar{x}) - g_{\hat{\mu}\hat{\nu}}(p) | \leq C \epsilon^{\alpha/2} |\bar{x}|^\alpha,
\end{equation}
by which we infer \eqref{transition approximate scaling}. Since $\alpha>0$, it is clear that the last term becomes negligible for small $\epsilon$, and the claim of unchanging signature also follows.
\end{proof}
To summarize, we have equipped $\tilde{B}$ with a metric 
\begin{equation}\label{general connected infinitesimal metric}
\tilde{g}_\epsilon =  \begin{dcases}
	\psi^* \bar{g}_\epsilon & \text{on } \tilde{U} \\
  	\epsilon f \tilde{g} + (1-f) \psi^*\bar{g}_\epsilon & \text{on } \tilde{W}\setminus\tilde{U} \\
  	\epsilon\tilde{g} & \text{on } \tilde{B}\setminus\tilde{W}
  \end{dcases},
\end{equation}
which is an element of $D_\EH(\tilde{B})$, and is such that the new topology in the central region scales by a factor of $\epsilon$ and the outer collar matches the geometry of shrinking balls in the background manifold $(M,g)$. In view of formulas \eqref{scaling metric} and \eqref{transition approximate scaling}, it is clear that the whole construction behaves as an approximate scaling by $\epsilon$, i.e. the leading order terms are $O(\epsilon)$. This additionally shows that the generic setup, despite technical complications, closely resembles the flat case. The gluing can now take place without problems. The general version of the infinitesimal topological variation is given by
\begin{align}
\begin{split}\label{general connected infinitesimal variation}
M &\longrightarrow (M\setminus \hat{B}_{\sqrt{\epsilon}}(p))\sqcup \tilde{B} / \phi_\epsilon\circ\psi \\
g &\longrightarrow g|_{M\setminus \hat{B}_{\sqrt{\epsilon}}(p)} \oplus \tilde{g}_\epsilon \\
\Omega &\longrightarrow (\Omega\setminus B_{\sqrt{\epsilon}}(p))\sqcup \tilde{B} / \phi_\epsilon\circ\psi
\end{split}.
\end{align}
Note that this depends on all introduced features, including not only the new metric $\tilde{g}$, but also the auxiliary metric $\hat{g}$ (which determines the shrinking balls), the bump function $f$ (which affects the transition between the central and outer region of the variation), as well as the choice of transition regions $\tilde{W}$ and $\tilde{U}$ and the diffeomorphism $\psi$.

\hfill

\paragraph{\textbf{Topology on $\D_\EH$.}} The topology of $\D_\EH$ can be adjusted in the spirit of disconnected topological variations, by extending the disjoint union topology in a way that forces deformation maps to be continuous. To give a rough outline of this, we remark that the main difference now is that the topology should be allowed to vary ``infinitesimally'' inside $\Omega$ as in \eqref{general connected infinitesimal variation} around a finite number of points. The metric can also be varied as per usual; note however that now extra care is warranted as this interferes with the gluing. This contains as a special case the variation of metric alone without the addition of new topology. Note that disconnected topological variations could also be included in this in a unified scheme, but since they have already been covered in depth in the previous section, we will refrain from needlessly carrying around terms related to them.

To make things precise, let us introduce some notions. 
\begin{definition}
Let $(M,g;\Omega)\in \D_\EH$ be a variational configuration, and let $\hat{g}$ be an auxiliary smooth Riemannian metric on $M$. A \textit{$k$-template $\Theta$ subordinate to $\hat{g}$} is a finite set consisting of the following surgery data:
\begin{itemize}
\item a finite set $\{p_i\}_{i=1}^k\subset \Omega$ of $k$ points,
\item  a finite set $\{ (\tilde{B}_i,\tilde{g}_i) \}_{i=1}^k$ of $k$ manifolds $\tilde{B}_i$ with boundaries $\partial \tilde{B}_i$ diffeomorphic to $\partial \B$, and metrics $\tilde{g}_i \in D_\EH(\tilde{B}_i)$,
\item a finite set $\{ (\tilde{U}_i, \tilde{W}_i, \psi_i) \}_{i=1}^k$ of $k$ nested collar pairs $\tilde{U}_i$ and $\tilde{W}_i$ with $\partial \tilde{B}_i \subset \tilde{U}_i \subset \tilde{W}_i$ and diffeomorphisms 
\begin{equation}
\psi_i: \tilde{W}_i \rightarrow \B\setminus B(0,\epsilon_{1i}), \quad i=1,\ldots,k,
\end{equation}
which restrict to diffeomorphisms $\psi_i:\tilde{U}_i \rightarrow \B\setminus B(0,\epsilon_{2i})$ where $0<\epsilon_{1i}<\epsilon_{2i}<1$, and also satisfy $\psi_i(\partial \tilde{B}_i) = \partial \B$.
\item a finite set $\{ f_i: \tilde{B}_i \rightarrow [0,1] \}_{i=1}^k$ of $k$ smooth cutoff functions such that $f_i = 1$ in $\tilde{B}_i\setminus \tilde{W}_i$ and $\supp f_i \subset \tilde{B}_i\setminus\tilde{U}_i$.
\end{itemize}
The pair $(\Theta,\hat{g})$ will be called a \textit{$k$-template}, or just \textit{template} if the number of points is irrelevant.
\end{definition}

\noindent
A template $(\Theta,\hat{g})$ contains all the data required to perform the surgery according to the instructions \eqref{general connected infinitesimal variation} around a finite number of points $p_i$ simultaneously, each with an independent deformation parameter $\epsilon_i$. We do not extend the notion to an infinite number of points as this might destroy compactness; remember that we are interested in variations with compact support. Note that since the points are finitely many, there exists an $\epsilon_0=\epsilon_0(\Omega,\Theta,\hat{g})>0$ such that
\begin{equation}
\hat{B}_{\sqrt{\epsilon_0}}(p_i) \subset\subset \Omega \text{ for } i=1,\ldots,k, \text{ and } \hat{B}_{\sqrt{\epsilon_0}}(p_i) \cap \hat{B}_{\sqrt{\epsilon_0}}(p_j) = \varnothing \text{ for } i\neq j,
\end{equation}
and so the surgery is performed without problems for any value of the deformation parameter $k$-vector $\epsilon=(\epsilon_1,\ldots,\epsilon_k)\in [0,\epsilon_0)^k$. As mentioned earlier, for an arbitrary choice of the $\tilde{g}_i$ the resulting manifold may have a degenerate metric, a situation we would like to avoid. 
\begin{definition}
The templates $(\Theta,\hat{g})$ that specify a surgery pattern in which the resulting metrics $\widetilde{g_i}_{\epsilon_i}$, given by \eqref{general connected infinitesimal metric} for each of the $\tilde{B}_i$, have the same signature as $g$ for all suitably small values of the deformation parameter vector $\epsilon\in[0,\epsilon_0)^k$ will be called \textit{signature preserving}.
\end{definition}
\noindent
Note that in view of Lemma \ref{linear interpolation lemma}, many such templates exist and are easy to construct for $\epsilon_0$ small enough.

\begin{definition}
Given a variational configuration $(M,g;\Omega)\in \D_\EH$ and a signature preserving $k$-template $(\Theta,\hat{g})$, we may define the \textit{special deformation map}
\begin{equation}\label{special deformation map}
\Phi_{(M,g;\Omega,\Theta,\hat{g})}: [0,\epsilon_0)^k \rightarrow \D_\EH,\quad \Phi_{(M,g;\Omega,\Theta,\hat{g})}(\epsilon) = (M_\epsilon,g_\epsilon;\Omega_\epsilon),
\end{equation}
where
\begin{equation}
M_\epsilon = \Big( M\setminus \bigcup_{i\in I(\epsilon)} \hat{B}_{\sqrt{\epsilon_i}}(p_i) \Big) \sqcup \Big( \bigsqcup_{i\in I(\epsilon)} \tilde{B}_i \Big) / \{ \phi_{\epsilon_i}\circ\psi_i : i\in I(\epsilon) \},
\end{equation}
\begin{equation}
g_\epsilon = g|_{M\setminus \bigcup_{i\in I(\epsilon)} \hat{B}_{\sqrt{\epsilon_i}}(p_i)} \oplus \bigoplus_{i\in I(\epsilon)} \tilde{g}_{\epsilon_i},
\end{equation}
\begin{equation}
\Omega_\epsilon = \Big( \Omega\setminus \bigcup_{i\in I(\epsilon)} \hat{B}_{\sqrt{\epsilon_i}}(p_i) \Big) \sqcup \Big( \bigsqcup_{i\in I(\epsilon)} \tilde{B}_i \Big) / \{ \phi_{\epsilon_i}\circ\psi_i : i\in I(\epsilon) \},
\end{equation}
where $I(\epsilon) = \{ i\in \{1,\ldots,k\}: \epsilon_i \neq 0 \}$, performing the surgery as specified above according to the value of the deformation parameters $\epsilon_i,\ i=1,\ldots, k$.
\end{definition}

Note that $\Phi_{(M,g;\Omega,\Theta,\hat{g})}(0,\ldots,0) = (M,g;\Omega)$, so the map indeed describes a deformation of the initial variational configuration. Special deformation maps describe pure connected topological variations. As with the disconnected case, we also want to include in this setup mixed variations of both the geometry and topology. With disconnected topological variations this was easy as there was a clear separation between the two; the supports of the two variations belonged to different manifolds. Here the situation is more complicated, but progress is still possible. If we wish to also vary the metric away from the points $p_i,\ i=1,\ldots,k,$ and with compact support in $\Omega$, we may proceed as usual by considering a perturbation $g\rightarrow g+h$, where $h\in W^{2,p}_c(\operatorname{Sym} \Omega)$. Because of the embedding $W^{2,p}_\loc(M)\hookrightarrow C^{0,\alpha}_\loc(M)$, an $h$ that has small Sobolev norm will also be pointwise small in the usual sense, so there exists $\epsilon_0'>0$ such that 
\begin{equation}\label{signature preserving condition}
\| h \|_{W^{2,p}(\Omega,\hat{g})} < \epsilon_0' \quad \Rightarrow \quad g+h \in D_\EH(M).
\end{equation}
It follows that if $(\Theta,\hat{g})$ is signature preserving with respect to $g$ and the Sobolev norm of $h$ is small enough, the deformation map $\Phi_{(M,g+h;\Omega,\Theta,\hat{g})}$ is still signature preserving, possibly for a smaller value of $\epsilon_0$. This means that there exists a neighborhood $\mathcal{U}(M,g;\Omega,\Theta,\hat{g})$ of $0\in W^{2,p}_c(\operatorname{Sym} \Omega)$ such that the map
\begin{equation}\label{generic deformation map}
\Phi_{(M,g;\Omega,\Theta,\hat{g})}: \mathcal{U}(M,g;\Omega,\Theta,\hat{g}) \times [0,\epsilon_0)^k \rightarrow \D_\EH,\quad (h,\epsilon)\mapsto \Phi_{(M,g+h;\Omega,\Theta,\hat{g})}(\epsilon)
\end{equation}
is well-defined and signature preserving. Note that purely geometric variations, of the metric alone,  are recovered as the special case $\Theta = \varnothing$.

\begin{definition}
\noindent
\begin{enumerate}
\item Any map of the form \eqref{generic deformation map} will be called a \textit{generic deformation map}. We denote by $\mathcal{F}(M,g;\Omega)$ the family of all signature preserving generic deformation maps for $(M,g;\Omega) \in \D_\EH$, and also their union for all variational configurations
\begin{equation}
\mathcal{F}:= \bigcup \{ \mathcal{F}(M,g;\Omega) : (M,g;\Omega) \in \D_\EH\}.
\end{equation}
\item Let $\tau_2$ denote the final topology of $\D_\EH$ with respect to the family of signature preserving generic deformation maps $\mathcal{F}$, i.e.
\begin{equation}
\mathfrak{U}\subset \D_\EH \text{ is open \textit{if and only if} } \Phi^{-1}(\mathfrak{U}) \text{ is open for all } \Phi\in \mathcal{F}. 
\end{equation}
\end{enumerate}
\end{definition}

\begin{theorem}\label{Connected topological continuity thm}
Let $(M,g;\Omega)\in \D_\EH$ be a variational configuration, and $(\Theta,\hat{g})$ a signature preserving $k$-template. If $\Phi_{(M,g;\Omega,\Theta,\hat{g})}: \mathcal{U}(M,g;\Omega,\Theta,\hat{g}) \times [0,\epsilon_0)^k \rightarrow \D_\EH$ is the associated generic deformation map, then the function 
\begin{equation}\label{EH deformation composition}
S_\EH \circ \Phi_{(M,g;\Omega,\Theta,\hat{g})}: \mathcal{U}(M,g;\Omega,\Theta,\hat{g}) \times [0,\epsilon_0)^k \rightarrow \R
\end{equation}
is continuous if and only if $n>2$, or $n=2$ and 
\begin{equation}\label{n=2 condition}
\int_{\tilde{B}_i\setminus \tilde{W}_i} \mathcal{L}[\tilde{g}_i] + \int_{\tilde{W}_i\setminus \tilde{U}_i} \mathcal{L}[\psi_i^*\breve{g}_i] = 0 \quad \forall i=1,\ldots,k,
\end{equation}
where $(\breve{g}_i)_{\bar{\mu}\bar{\nu}}(\bar{x}) = f_i(\psi_{i*} \tilde{g}_i)_{\bar{\mu}\bar{\nu}}(\bar{x}) + (1-f_i)g_{\hat{\mu}\hat{\nu}}(p_i)$. In particular, $S_\EH:(\D_\EH, \tau_2)\rightarrow \R$ is continuous if and only if $n>2$.
\end{theorem}

\begin{remark}
Condition \eqref{n=2 condition} is required only when $n=2$ to ensure continuity, and geometrically corresponds to the vanishing of the mean scalar curvature of topological components in the limit $\epsilon_i\rightarrow 0$ for each $i\in \{1,\ldots,k\}$.
\end{remark}

\begin{proof}
Let us denote $S(h,\epsilon):= S_\EH \circ \Phi_{(M,g;\Omega,\Theta,\hat{g})} (h,\epsilon)$, and let $(h_*,\epsilon_*)\in \mathcal{U}(M,g;\Omega,\Theta,\hat{g}) \times [0,\epsilon_0)^k$. To show that $S$ is continuous at $(h_*,\epsilon_*)$, we need to show that it is continuous with respect to each of the variables separately. First, observe that
\begin{equation}\label{action split}
\begin{aligned}
S(h,\epsilon) &= \int_{\Omega \setminus \bigcup_{i\in I(\epsilon)} \hat{B}_{\sqrt{\epsilon_i}}(p_i)} \mathcal{L}[g+h] \\
	&\hspace{0.5cm}+ \sum_{i\in I(\epsilon)} \int_{\tilde{B}_i} \mathcal{L}[\widetilde{g+h}_{\epsilon_i}] \\
	&= \int_{\Omega \setminus \bigcup_{i\in I(\epsilon)} \hat{B}_{\sqrt{\epsilon_i}}(p_i)} \mathcal{L}[g+h] &\quad =:I_0(h,\epsilon)\\
	&\hspace{0.5cm} + \sum_{i\in I(\epsilon)} \int_{\tilde{U}_i} \mathcal{L}[\psi_i^* \overline{g+h}_{\epsilon_i}] &\quad =:I_1(h,\epsilon)\\
	&\hspace{1cm} +\sum_{i\in I(\epsilon)} \int_{\tilde{W}_i\setminus \tilde{U}_i} \mathcal{L}[\epsilon_i f_i \tilde{g}_i + (1-f_i) \psi_i^* \overline{g+h}_{\epsilon_i}] &\quad =:I_2(h,\epsilon) \\
	&\hspace{1.5cm} +\sum_{i\in I(\epsilon)} \int_{\tilde{B}_i\setminus\tilde{W}_i} \mathcal{L}[\epsilon_i \tilde{g}_i] &\quad =:I_3(h,\epsilon)
\end{aligned}.
\end{equation}

\paragraph{\textbf{Case 1} (Continuity with respect to geometric perturbations and fixed deformation parameters).} Fix $\epsilon = \epsilon_*$. Then the surgery pattern and corresponding post-surgery manifold $M_*$ and domain $\Omega_*$ are specified by $\Theta$, $\hat{g}$ and $\epsilon_*$. The maps $h\mapsto g+h$, $h\mapsto \psi_i^* \overline{g+h}_{\epsilon_{*i}}$ and $h\mapsto \epsilon_{*i} f_i \tilde{g}_i + (1-f_i) \psi_i^* \overline{g+h}_{\epsilon_{*i}}$ are all continuous from $\mathcal{U}(M,g;\Omega,\Theta,\hat{g}) \subset W^{2,p}_c(\operatorname{Sym} \Omega)$ to $W^{2,p}(\Lz \Omega_*)$ for each $i=1,\ldots,k$. In view of Theorem \ref{Geometric continuity thm} we deduce that
\begin{equation}
I_j(h,\epsilon_*)\rightarrow I_j(h_*,\epsilon_*)\quad \text{as}\quad h\rightarrow h_* \quad \text{in}\quad \mathcal{U}(M,g;\Omega,\Theta,\hat{g})
\end{equation}
for each $j=0,1,2$, and $I_3(h,\epsilon)$ is independent of $h$. It follows that 
\begin{equation}
S(h,\epsilon_*)\rightarrow S(h_*,\epsilon_*)\quad \text{as}\quad h\rightarrow h_* \quad \text{in}\quad \mathcal{U}(M,g;\Omega,\Theta,\hat{g}).
\end{equation}

\paragraph{\textbf{Case 2} (Continuity with respect to deformation parameters and fixed background geometry).} Now fix $h=h_*$, and set $g_*=g+h_*$. To demonstrate that $S(h_*,\epsilon) \rightarrow S(h_*,\epsilon_*)$ as $\epsilon\rightarrow \epsilon_*$, it suffices to show that $I_j(h_*,\epsilon) \rightarrow I_j(h_*,\epsilon_*)$ as $\epsilon \rightarrow \epsilon_*$ for each $j=0,1,2,3$. For $j=0$, we have that
\begin{equation}
\sum_{i\in I(\epsilon)\cap I(\epsilon_*)} \bigg| \int_{\hat{B}_{\sqrt{\epsilon_i}}(p_i)} \mathcal{L}[g_*] - \int_{\hat{B}_{\sqrt{\epsilon_{*i}}}(p_i)} \mathcal{L}[g_*] \, \bigg| \leq \sum_{i\in I(\epsilon)\cap I(\epsilon_*)} \int_\Omega | \mathcal{L}[g_*] | \chi_{\hat{B}_{\sqrt{\epsilon_i}}(p_i) \bigtriangleup \hat{B}_{\sqrt{\epsilon_{*i}}}(p_i)},
\end{equation}
and likewise
\begin{equation}
\sum_{i\in I(\epsilon)\setminus I(\epsilon_*)} \bigg| \int_{\hat{B}_{\sqrt{\epsilon_i}}(p_i)} \mathcal{L}[g_*] \, \bigg| \leq \sum_{i\in I(\epsilon)\setminus I(\epsilon_*)} \int_\Omega | \mathcal{L}[g_*] | \chi_{\hat{B}_{\sqrt{\epsilon_i}}(p_i)},
\end{equation}
where $\chi_A$ stands for the characteristic function of a set $A$ and $A \bigtriangleup B = (A\setminus B) \cup (B\setminus A)$ stands for the symmetric difference of two sets $A$ and $B$. Since 
\begin{equation}
\lim_{\epsilon\rightarrow\epsilon_*} \chi_{\hat{B}_{\sqrt{\epsilon_i}}(p_i) \bigtriangleup \hat{B}_{\sqrt{\epsilon_{*i}}}(p_i)} = 0 \quad a.e \quad \text{for } i\in I(\epsilon)\cap I(\epsilon_*)
\end{equation}
and 
\begin{equation}
 \lim_{\epsilon\rightarrow\epsilon_*} \chi_{\hat{B}_{\sqrt{\epsilon_i}}(p_i)} = 0 \quad a.e. \quad \text{for } i\in I(\epsilon)\setminus I(\epsilon_*)
\end{equation}
in $\Omega$, it follows by the dominated convergence theorem that 
\begin{equation}
\lim_{\epsilon\rightarrow\epsilon_*} |I_0(h_*,\epsilon)-I_0(h_*,\epsilon_*)| = 0.
\end{equation}

For $j=1$, the situation is quite similar, as the metric in the outer collars $\tilde{U}_i$ is still a pullback of $g_*$. In particular, we have that
\begin{equation}
\begin{aligned}
&\sum_{i\in I(\epsilon)\cap I(\epsilon_*)} \bigg| \int_{\tilde{U}_i} \mathcal{L}[\psi_i^* \overline{g_*}_{\epsilon_i}] - \int_{\tilde{U}_i} \mathcal{L}[\psi_i^* \overline{g_*}_{\epsilon_{*i}}]\, \bigg| \\
	&\hspace{3cm} = \sum_{i\in I(\epsilon)\cap I(\epsilon_*)} \bigg| \int_{\phi_{\epsilon_i}(\U_i)} \mathcal{L}[g_*] - \int_{\phi_{\epsilon_{*i}}(\U_i)} \mathcal{L}[g_*]\, \bigg| \\
	&\hspace{6cm} \leq \sum_{i\in I(\epsilon)\cap I(\epsilon_*)} \int_\Omega |\mathcal{L}[g_*]| \chi_{\phi_{\epsilon_i}(\U_i) \bigtriangleup \phi_{\epsilon_{*i}}(\U_i)},
\end{aligned}
\end{equation}
where $\U_i := \psi_i(\tilde{U}_i) = \B\setminus B(0,\epsilon_{2i})$ is the outer collar in $\B$ for each $i\in I(\epsilon)\cap I(\epsilon_*)$. Since each $\U_i$ is an annulus centered at $0\in \B$, there easily follows that 
\begin{equation}
\lim_{\epsilon\rightarrow \epsilon_*} \chi_{\phi_{\epsilon_i}(\U_i) \bigtriangleup \phi_{\epsilon_{*i}}(\U_i)} = 0 \quad a.e. \quad \text{for } i\in I(\epsilon)\cap I(\epsilon_*)
\end{equation}
in $\Omega$, and the sum above converges to zero in view of the dominated convergence theorem. The sum over $i \in I(\epsilon)\setminus I(\epsilon_*)$ is similar. This proves that
\begin{equation}
\lim_{\epsilon\rightarrow\epsilon_*} |I_1(h_*,\epsilon)-I_1(h_*,\epsilon_*)| = 0.
\end{equation}

It remains to examine convergence of $I_2(h_*,\epsilon)$ and $I_3(h_*,\epsilon)$. Note that if $\epsilon_{*i}\neq 0$, the maps $\epsilon_i \mapsto \epsilon_i \tilde{g}_i$ and $\epsilon_i\mapsto \epsilon_i f_i \tilde{g}_i + (1-f_i) \psi_i^* \overline{g_*}_{\epsilon_i}$ are continuous from $(0,\epsilon_0)$ to $W^{2,p}(\Lz \tilde{B}_i)$ as $\epsilon_i \rightarrow \epsilon_{*i}$ in view of Lemma \ref{Contraction Lemma}, so by Theorem \ref{Geometric continuity thm} it follows that 
\begin{equation}
\lim_{\epsilon\rightarrow\epsilon_*} \sum_{i\in I(\epsilon)\cap I(\epsilon_*)} \bigg| \int_{\tilde{W}_i\setminus\tilde{U}_i} \mathcal{L}[\epsilon_i f_i \tilde{g}_i +(1-f_i) \psi_i^* \overline{g_*}_{\epsilon_i}] - \int_{\tilde{W}_i\setminus\tilde{U}_i} \mathcal{L}[\epsilon_{*i} f_i \tilde{g}_i +(1-f_i) \psi_i^* \overline{g_*}_{\epsilon_{*i}}] \,\bigg| =0
\end{equation}
and 
\begin{equation}
\lim_{\epsilon\rightarrow\epsilon_*} \sum_{i\in I(\epsilon)\cap I(\epsilon_*)} \bigg| \int_{\tilde{B}_i\setminus\tilde{W}_i} \mathcal{L}[\epsilon_i \tilde{g}_i] - \int_{\tilde{B}_i\setminus\tilde{W}_i} \mathcal{L}[\epsilon_{*i} \tilde{g}_i] \,\bigg| =0.
\end{equation}

It should be stressed that this argument is invalid when $\epsilon_{*i} = 0$, as the (zero) limit of $\epsilon_i \tilde{g}_i$ is no longer an element of $W^{2,p}(\Lz \tilde{B}_i)$ and neither Lemma \ref{Contraction Lemma} nor Theorem \ref{Geometric continuity thm} apply any more. In that case we have to instead compute the integrals by hand. For $j=3$ the calculation is straightforward, as we have already demonstrated the scaling property
\begin{equation}
\mathcal{L}[\epsilon_i \tilde{g}_i] = \epsilon_i^\frac{n-2}{2} \mathcal{L}[\tilde{g}_i]
\end{equation}
in Lemma \ref{Scaling Lemma}, which implies $\mathcal{L}[\epsilon_i \tilde{g}_i] \rightarrow 0$ in $L^1(\Lambda^n\tilde{B}_i\setminus \tilde{W}_i)$ if and only if $n>2$ or $n=2$ and $\tilde{g}_i$ has zero mean scalar curvature in $\tilde{B}_i\setminus\tilde{W}_i$, by linearity of the integral. This shows that we also have
\begin{equation}
\lim_{\epsilon\rightarrow\epsilon_*} \sum_{i\in I(\epsilon)\setminus I(\epsilon_*)} \bigg| \int_{\tilde{B}_i\setminus\tilde{W}_i} \mathcal{L}[\epsilon_i \tilde{g}_i] \,\bigg| =0 \quad \text{ if } n>2,
\end{equation}
and consequently
\begin{equation}
\lim_{\epsilon\rightarrow\epsilon_*} |I_3(h_*,\epsilon)-I_3(h_*,\epsilon_*)| = 0 \quad \text{ if } n>2.
\end{equation}

The case $j=2$ is similar but requires a more delicate argument. In Lemma \ref{linear interpolation lemma}, we have shown that the push-forward of the metric $\widetilde{g_*}_{\epsilon_i}$ in the transition region $\W\setminus\U$ looks like
\begin{equation}
(\psi_{i*} \widetilde{g_*}_{\epsilon_i})_{\bar{\mu}\bar{\nu}}(\bar{x}) = \epsilon_i (\breve{g}_{\bar{\mu}\bar{\nu}}(\bar{x}) + \gamma_{\bar{\mu}\bar{\nu}}(\epsilon_i^{1/2}\bar{x})),
\end{equation}
where $\breve{g}_{\bar{\mu}\bar{\nu}}(\bar{x}) = f_i(\psi_{i*} \tilde{g}_i)_{\bar{\mu}\bar{\nu}}(\bar{x}) + (1-f_i)g_{\hat{\mu}\hat{\nu}}(p_i)$ and $|\gamma_{\bar{\mu}\bar{\nu}}(\bar{y})|\leq C |\bar{y}|^\alpha$, by which there follows
\begin{equation}
|(\psi_{i*} \widetilde{g_*}_{\epsilon_i})_{\bar{\mu}\bar{\nu}}(\bar{x}) - \epsilon_i \breve{g}_{\bar{\mu}\bar{\nu}}(\bar{x}) | \leq C \epsilon_i^{1+\alpha/2}|\bar{x}|^\alpha,
\end{equation}
and so 
\begin{equation}
(\psi_{i*} \widetilde{g_*}_{\epsilon_i})_{\bar{\mu}\bar{\nu}}(\bar{x}) = \epsilon_i \breve{g}_{\bar{\mu}\bar{\nu}}(\bar{x}) + o(\epsilon_i),
\end{equation}
By differentiating we also obtain the weak derivatives
\begin{equation}
\partial_{\bar{\rho}} (\psi_{i*} \widetilde{g_*}_{\epsilon_i})_{\bar{\mu}\bar{\nu}}(\bar{x}) = \epsilon_i \partial_{\bar{\rho}} \breve{g}_{\bar{\mu}\bar{\nu}}(\bar{x}) + \epsilon_i^{3/2} \partial_{\bar{\rho}} \gamma_{\bar{\mu}\bar{\nu}}(\epsilon_i^{1/2}\bar{x}),
\end{equation}
and the second weak derivatives
\begin{equation}
\partial_{\bar{\sigma}}\partial_{\bar{\rho}} (\psi_{i*} \widetilde{g_*}_{\epsilon_i})_{\bar{\mu}\bar{\nu}}(\bar{x}) = \epsilon_i \partial_{\bar{\sigma}}\partial_{\bar{\rho}} \breve{g}_{\bar{\mu}\bar{\nu}}(\bar{x}) + \epsilon_i^2 \partial_{\bar{\sigma}}\partial_{\bar{\rho}} \gamma_{\bar{\mu}\bar{\nu}}(\epsilon_i^{1/2}\bar{x}),
\end{equation}
noting that, in all orders, leading contributions are $O(\epsilon_i)$ and depend only on $\breve{g}$. Inserting these into the definition of scalar curvature \eqref{scalar curvature} and expanding in terms of $\epsilon_i$, we conclude that
\begin{equation}
R[\psi_{i*} \widetilde{g_*}_{\epsilon_i}] = \epsilon_i^{-1} R[\breve{g}] + o(\epsilon_i^{-1}),
\end{equation}
while
\begin{equation}
| \det (\psi_{i*} \widetilde{g_*}_{\epsilon_i})_{\bar{\mu}\bar{\nu}})\,|^{1/2} = \epsilon_i^{n/2} |\det\breve{g}_{\bar{\mu}\bar{\nu}}\,|^{1/2} + o(\epsilon_i^{n/2}).
\end{equation}
It follows that 
\begin{equation}
\mathcal{L}[\psi_{i*} \widetilde{g_*}_{\epsilon_i}] = \epsilon_i^{\frac{n-2}{2}} \mathcal{L}[\breve{g}] + o(\epsilon_i^{\frac{n-2}{2}}),
\end{equation}
and consequently $\mathcal{L}[\psi_{i*} \widetilde{g_*}_{\epsilon_i}]\rightarrow 0$ in $L^1(\Lambda^n \B)$ if and only if $n>2$ or $n=2$ and $\breve{g}$ has zero mean scalar curvature. This shows that
\begin{equation}
\lim_{\epsilon\rightarrow\epsilon_*} \sum_{i\in I(\epsilon)\setminus I(\epsilon_*)} \bigg| \int_{\tilde{W}_i\setminus\tilde{U}_i} \mathcal{L}[\epsilon_i f_i \tilde{g}_i +(1-f_i) \psi_i^* \overline{g_*}_{\epsilon_i}]\, \bigg| = 0 \quad \text{ if } n>2,
\end{equation}
which proves that 
\begin{equation}
\lim_{\epsilon\rightarrow\epsilon_*} |I_2(h_*,\epsilon)-I_2(h_*,\epsilon_*)| = 0 \quad \text{ if } n>2.
\end{equation}
For $n=2$, setting 
\begin{equation}
\kappa_i(M,g;\Omega,\Theta,\hat{g}) := \int_{\tilde{B}_i\setminus \tilde{W}_i} \mathcal{L}[\tilde{g}_i] + \int_{\tilde{W}_i\setminus \tilde{U}_i} \mathcal{L}[\psi_i^*\breve{g}_i],
\end{equation}
it is evident from \eqref{action split} that
\begin{equation}
S(0,\epsilon)-S(0,\epsilon_*) = \sum_{i\in I(\epsilon)\setminus I(\epsilon_*)} \kappa_i(M,g;\Omega,\Theta,\hat{g}) + o(1) \quad \text{ as } \epsilon\rightarrow \epsilon_*.
\end{equation}
This completes the proof.
\end{proof}

\paragraph{\textbf{Variation of Einstein-Hilbert action.}} We are now ready to define the topological functional derivative associated with connected topological variations.

\begin{definition}
\noindent
\begin{enumerate}
\item Let $(M,g;\Omega)\in \D_\EH$ be a given variational configuration, and let $(\Theta,\hat{g})$ be a signature preserving $1$-template with associated special deformation map $\Phi_{(M,g;\Omega,\Theta,\hat{g})}: [0,\epsilon_0)\rightarrow \D_\EH$. Assume, moreover, that $F : (\D_\EH,\tau_2) \rightarrow \R$ is a continuous functional. The \textit{connected topological functional derivative of $F$ at $(M,g;\Omega)$ along $(\Theta,\hat{g})$} is 
\begin{equation}
\delta_{(M,g;\Omega)} F [\Theta,\hat{g}] := \frac{d}{d \epsilon} \bigg|_{\epsilon=0^+} F \circ \Phi_{(M,g;\Omega,\Theta,\hat{g})}(\epsilon),
\end{equation}
provided that the one-sided derivative exists.
\item If the connected topological functional derivative exists for all $(M,g;\Omega)$ and all $(\Theta,\hat{g})$, we say that $F$ is \textit{differentiable with respect to connected topological variations}.
\end{enumerate}
\end{definition}

Finally we have the following version of Theorem \ref{disjoint EH theorem} for connected topological variations.

\begin{theorem}\label{connected EH theorem}
Let $S_\EH: (\D_\EH,\tau_2) \rightarrow \R$ be the Einstein-Hilbert action extended to include connected topological variations of compact support, and let $n > 2$. Then the following statements are true.
\begin{enumerate}
\item If $n<4$, the action is not differentiable with respect to connected topological variations.
\item If $n=4$, the action is differentiable almost everywhere and the connected topological functional derivative is not identically zero.
\item If $n>4$, the action is differentiable almost everywhere and the connected topological functional derivative is identically zero. 
\end{enumerate}
Specifically, ``almost everywhere'' coincides with the set of Lebesgue points of scalar curvature for each given variational configuration.
\end{theorem}

\begin{proof}
Let $0<\epsilon<\epsilon_0$ and $S(\epsilon):= S_\EH\circ\Phi_{(M,g;\Omega,\Theta,\hat{g})}(\epsilon)$. Then setting $h=0$ and $k=1$ in \eqref{action split} from the proof of Theorem \ref{Connected topological continuity thm}, we see that
\begin{equation}
\begin{aligned}
S(\epsilon)-S(0) &= -\int_{\hat{B}_{\sqrt{\epsilon}}(p)} \mathcal{L}[g] &\quad =:I_0(\epsilon)\\
	&\hspace{0.5cm} + \int_{\tilde{U}} \mathcal{L}[\psi^* \bar{g}_{\epsilon}] &\quad =:I_1(\epsilon)\\
	&\hspace{1cm} + \int_{\tilde{W}\setminus \tilde{U}} \mathcal{L}[\epsilon f \tilde{g} + (1-f) \psi^* \bar{g}_{\epsilon}] &\quad =:I_2(\epsilon) \\
	&\hspace{1.5cm} + \int_{\tilde{B}\setminus\tilde{W}} \mathcal{L}[\epsilon \tilde{g}] &\quad =:I_3(\epsilon)
\end{aligned}.
\end{equation}
In view of the Lebesgue differentiation theorem, we have that
\begin{equation}
\lim_{\epsilon\rightarrow 0} \frac{1}{\operatorname{vol}_{\hat{g}}(\hat{B}_{\sqrt{\epsilon}}(p))} \int_{\hat{B}_{\sqrt{\epsilon}}(p)} \mathcal{L}[g] = R[g](p)\frac{|\det g_{\mu\nu}(p)|^{1/2}}{|\det \hat{g}_{\mu\nu}(p)|^{1/2}}
\end{equation}
for almost every $p\in \Omega$. Since 
\begin{equation}
\operatorname{vol}_{\hat{g}}(\hat{B}_{\sqrt{\epsilon}}(p)) = \omega_n \epsilon^{n/2} +o(\epsilon^{n/2}),
\end{equation}
with $\omega_n$ being the volume of the Euclidean ball $\B$ of radius $1$, it follows that
\begin{equation}
I_0(\epsilon) = -\omega_n R[g](p)\frac{|\det g_{\mu\nu}(p)|^{1/2}}{|\det \hat{g}_{\mu\nu}(p)|^{1/2}} \epsilon^{n/2} + o(\epsilon^{n/2}).
\end{equation}
Similarly, since $\psi(\tilde{U}) = \B\setminus B(0,\epsilon_2)$ is an annulus for some $0<\epsilon_2<1$, there holds
\begin{equation}
I_1(\epsilon) = \omega_n R[g](p)\frac{|\det g_{\mu\nu}(p)|^{1/2}}{|\det \hat{g}_{\mu\nu}(p)|^{1/2}}(1-\epsilon_2^{n/2})\epsilon^{n/2} + o(\epsilon^{n/2}),
\end{equation}
and in Lemma \ref{Contraction Lemma} and the proof of Theorem \ref{Connected topological continuity thm} we have already shown that
\begin{equation}
I_2(\epsilon) = \epsilon^\frac{n-2}{2} \int_{\W\setminus\U} \mathcal{L}[\breve{g}] + o(\epsilon^\frac{n-2}{2})
\end{equation}
where $\breve{g}_{\bar{\mu}\bar{\nu}}(\bar{x}) = f(\psi_{*} \tilde{g})_{\bar{\mu}\bar{\nu}}(\bar{x}) + (1-f)g_{\hat{\mu}\hat{\nu}}(p)$, and 
\begin{equation}
I_3(\epsilon) = \epsilon^\frac{n-2}{2} \int_{\tilde{B}\setminus\tilde{W}} \mathcal{L}[\tilde{g}].
\end{equation}
It follows that
\begin{equation}
S(\epsilon)-S(0)= \kappa(M,g;\Omega,\Theta,\hat{g}) \epsilon^\frac{n-2}{2} + o(\epsilon^\frac{n-2}{2}),
\end{equation}
where 
\begin{equation}
\kappa(M,g;\Omega,\Theta,\hat{g}) =  \int_{\tilde{B}\setminus\tilde{W}} \mathcal{L}[\tilde{g}] + \int_{\tilde{W}\setminus\tilde{U}} \mathcal{L}[\psi^*\breve{g}]
\end{equation}
is a quantity associated with the mean scalar curvature of the variation in the limit $\epsilon\rightarrow 0^+$. From this we conclude that generically
\begin{equation}
\delta_{(M,g;\Omega)}S_\EH[\Theta,\hat{g}] = \lim_{\epsilon\rightarrow 0^+} \frac{S(\epsilon)-S(0)}{\epsilon} = \begin{dcases}
\text{undefined} & \text{if } n<4 \\
\kappa(M,g;\Omega,\Theta,\hat{g}) & \text{if } n=4 \\
0 &\text{if } n>4
\end{dcases}.
\end{equation}
This completes the proof.
\end{proof}

This confirms that the Einstein-Hilbert action behaves in the same way under connected and disconnected topological variations. In particular, the critical dimension $n=4$ is the same and in this case the action possesses no critical points as the quantity $\kappa(M,g;\Omega,\Theta,\hat{g})$ is generically non-zero.

\begin{remark}
It should be noted that under stronger regularity assumptions it is possible to promote ``almost everywhere'' statements to ``everywhere'' statements in the above theorem. For example, this is the case for $C^2$ metrics for which scalar curvature is continuous and hence every point is a Lebesgue point. But even in the case of Sobolev spaces, this can be achieved via direct estimates of terms $I_0(\epsilon)$ and $I_1(\epsilon)$ via the H\"older inequality, when $p$ is large enough. We omit the details.
\end{remark}

\begin{remark}[Comparison with ``punctured space'' derivatives] A version of the topological derivative that relies on simpler means than our proposed gluing setup has been proposed and is more appropriate for different types of problems. In this approach, the associated deformation is obtained by removing a shrinking ball or more complicated set whose radius goes to zero about some point, in which case the limit is locally a punctured space; see, e.g., Amstutz \cite{amstutz2021} and references therein. The difference in action with respect to this deformation then corresponds to the term $I_0(\epsilon)$ in the proof of Theorem \ref{connected EH theorem}, and more generally the term $I_0(\epsilon,h)$ which appears in the proof of Theorem \ref{Connected topological continuity thm}. Since we have demonstrated that the later term is continuous, Einstein-Hilbert action is also continuous under this type of variation. It is worth noting, however, that since $I_0(\epsilon) = O(\epsilon^{n/2})$, the critical dimension for these ``punctured space'' derivatives is $n=2$ instead of $n=4$. This means that, without extra effort, topological variations which come in the form of expanding ``cracks'' in spacetime can also be included in the theory; note that these would directly affect fundamental or singular homology groups and are indeed topological. 

This type of variation is naturally fitted to optimization problems in materials science, and is associated with the appearance of cracks or other defects. Nevertheless, in the setting of general relativity this comes with the additional complication of the appearance of arbitrarily many codimension-one boundaries in spacetime, and thus an uncontrolled amount of naked singularities in the sense of geodesic incompleteness. This scenario is quite unphysical at least in the classical (non-quantum) sense, and we will not pursue it further here. Note that in our gluing setup no arbitrary boundaries appear, as the ``cracks'' are filled by gluing other manifolds along the boundaries.
\end{remark}

\section{Functional blow-up of degenerating metrics} \label{Blow-up}

In this section we discuss an important geometric constraint arising in our analysis, which moreover crucially informs the choice of appropriate topological variations discussed earlier in the paper. In Section \ref{Disconnected topological variations} it was shown that, for a closed manifold $M'$, approaching the zero section in $W^{2,p}(\operatorname{Sym} M')$ along rays in $W^{2,p}(\Lz M')$ preserved continuity of the Einstein-Hilbert functional for $n>2$. While that approach was reasonable for introducing derivatives, it is natural to ask if this continues to be the case when one considers variations that are not restricted on rays, or more generally occur as the reverse-collapsing of lower dimensional objects other than points. What turns out to be the case is that Einstein-Hilbert is very sensitive in the way the boundary of $W^{2,p}(\Lz M')$ is approached, and that precise control of the degeneration is warranted to prevent catastrophic blow-ups.

Blow-up of curvature is a recurrent theme in geometric analysis, appearing for instance in finite-time singularity formation for geometric flows such as the Ricci flow \cite{cao2007,topping2010} and in bubbling or compactness breakdown for conformally invariant curvature problems such as the Yamabe equation \cite{lee1987,druet2005,brendle2007}. In \cite{nguyen2022}, Nguyen calculates the blow-up of scalar curvature for metrics collapsing along a distribution. In a similar spirit, we prove that in the following general case metrics collapsing along a submanifold can have diverging Einstein-Hilbert action, while derivatives of the metric up to any order remain uniformly bounded during the blow-up.

\begin{proposition}\label{Collapsing product}
Let $N$ be a closed smooth manifold of dimension $n\geq 2$, and let 
\begin{equation}
M:=S^1 \times N.
\end{equation}
There exists a family $\{ g_\epsilon: \epsilon\in (0,1]\}$ of smooth metrics of any given signature in $M$ such that
\begin{enumerate}
\item The $g_\epsilon$ are uniformly bounded in $C^\infty(\Lz M)$, and 
\item The Einstein-Hilbert functional evaluated at $g_\epsilon$ diverges, in particular
\begin{equation}
\int_M \mathcal{L}[g_\epsilon] = O(\epsilon^{-1}) \quad \text{as} \quad \epsilon\rightarrow 0^+.
\end{equation}
\end{enumerate}
\end{proposition}

\begin{proof}
We prove the result for Riemannian signature, as other signatures are only different up to a sign. Denote coordinates on $N$ by $x^a$, $a=1,\ldots,n$. Let $\gamma\in C^\infty(\Lz N)$ be a smooth Riemannian metric on $N$, and let $\theta\in C^\infty(\operatorname{Sym} M)$ be a smooth, symmetric, not identically vanishing $(0,2)$-tensor field such that $\gamma^{ab}\theta_{ab}=0$ everywhere on $N$. Denote by $t$ the coordinate on $S^1 = \R/2\pi\mathbb{Z}$ and let 
\begin{equation}
\gamma_{ab}(t)=\gamma_{ab} + \delta \sin(t) \theta_{ab},
\end{equation}
choosing $\delta>0$ small enough so that $\gamma(t)\in C^{\infty}(\Lz N)$ for all $t\in S^1$. In the product $M=S^1\times N$, define the family of smooth metrics
\begin{equation}
g_\epsilon = \epsilon^2 dt\otimes dt + \gamma_{ab}(t)dx^a \otimes dx^b.
\end{equation}
We will show that the $g_\epsilon$ have the desired properties.

The fact that $g_\epsilon$ are uniformly bounded in $C^\infty(\Lz M)$ is straightforward. Note that
\begin{equation}
(g_\epsilon)_{tt}=\epsilon^2, \qquad (g_\epsilon)_{ta}=0, \quad (g_\epsilon)_{ab} = \gamma_{ab}(t),
\end{equation}
so all derivatives of $(g_\epsilon)_{tt}$ vanish, while the same is true for $(g_\epsilon)_{ta}$ and its derivatives. Moreover, since $\sin(t)$ and all its derivatives are all bounded in absolute value by one, it is clear that the $\partial_\alpha (g_\epsilon)_{ab}$ are all uniformly bounded by constants that depend only on $\gamma$, $\theta$, $\delta$, and the multi-index $\alpha$. This proves item (1).

To prove item (2), first note that $M$ is naturally foliated by the hypersurfaces $N_t=\{t\}\times N$. Performing the standard $(n+1)$-split (see, e.g. \cite{giulini2014}), and noting that the metric is already expressed in coordinates where the lapse function is $\epsilon$ and the shift vector field is zero, the scalar curvature decomposition formula reads
\begin{equation}
R[g_\epsilon] = R[\gamma(t)] - K_{ab}(t)K^{ab}(t) + {K^a}_a(t)^2+ \nabla_\mu V^\mu,
\end{equation}
where $K_{ab}(t) = -\frac{1}{2\epsilon}\partial_t \gamma_{ab}(t)$ is the extrinsic curvature of the hypersurface $N_t$ embedded in $M$, indices are raised and lowered with respect to the metric $\gamma_{ab}(t)$ and $\nabla_\mu V^\mu$ is a divergence term which will be irrelevant in the integration. Since $\det (g_\epsilon)_{\mu\nu} = \epsilon^2 \det \gamma_{ab}(t)$, there follows
\begin{equation}
\int_M \mathcal{L}[g_\epsilon] = \int_{S^1}\int_{N_t} (R[\gamma(t)] - K_{ab}(t)K^{ab}(t) + {K^a}_a(t)^2)\epsilon |\det \gamma_{ab}(t)\,|^{1/2} \,dx\,dt.
\end{equation}
Observe that $R[\gamma(t)]|\det \gamma_{ab}(t)|^{1/2}$ is a smooth function on each $N_t$, which is moreover smooth as a function of $t$. Since both $N$ and $S^1$ are compact, we have that
\begin{equation}
c_1 := \int_{S^1}\int_{N_t} R[\gamma(t)] |\det \gamma_{ab}(t)\,|^{1/2} \,dx\,dt \in \R.
\end{equation}
Regarding the quadratic part of the integral, we easily obtain
\begin{equation}
K_{ab}(t) = -\frac{1}{2\epsilon} \frac{\partial}{\partial t} \gamma_{ab}(t) = -\frac{\delta \cos(t)}{2\epsilon}\theta_{ab},
\end{equation}
and consequently, 
\begin{equation}
{K^a}_a(t)=\gamma^{ab}(t)K_{ab}(t)=-\frac{\delta\cos(t)}{2\epsilon}\gamma^{ab}(t)\theta_{ab}
\end{equation}
and
\begin{equation}
K_{ab}(t)K^{ab}(t)=\gamma^{ac}(t)\gamma^{bd}(t)K_{ab}(t)K_{cd}(t)=\frac{\delta^2\cos^2(t)}{4\epsilon^2}\gamma^{ac}(t)\gamma^{bd}(t)\theta_{ab}\theta_{cd}.
\end{equation}
It follows that
\begin{equation}
{K^a}_a(t)^2-K_{ab}(t)K^{ab}(t) = \frac{\delta^2\cos^2(t)}{4\epsilon^2}((\gamma^{ab}(t)\theta_{ab})^2-\gamma^{ac}(t)\gamma^{bd}(t)\theta_{ab}\theta_{cd}).
\end{equation}
Taking into account that $\gamma^{ab}(t)=\gamma^{ab}+O(\delta)$ and that $\gamma^{ab}\theta_{ab}=0$, it follows that
\begin{equation}
(\gamma^{ab}(t)\theta_{ab})^2-\gamma^{ac}(t)\gamma^{bd}(t)\theta_{ab}\theta_{cd} = -\theta_{ab}\theta^{ab} + O(\delta).
\end{equation}
Choosing $\delta>0$ possibly even smaller, we can have
\begin{equation}
(\gamma^{ab}(t)\theta_{ab})^2-\gamma^{ac}(t)\gamma^{bd}(t)\theta_{ab}\theta_{cd} \leq 0
\end{equation}
and in fact strictly negative in a set of positive measure. Thus
\begin{equation}
c_{-1}=\frac{\delta^2}{4} \int_{S^1}\int_{N_t}\cos^2(t)((\gamma^{ab}(t)\theta_{ab})^2-\gamma^{ac}(t)\gamma^{bd}(t)\theta_{ab}\theta_{cd})|\det\gamma_{ab}(t)|^{1/2}\,dx\,dt \in \R_-.
\end{equation}
Taking all the above into account, we conclude that 
\begin{equation}
\int_M \mathcal{L}[g_\epsilon] = \frac{c_{-1}}{\epsilon}+c_1\epsilon.
\end{equation}
This completes the proof.
\end{proof}

We denote by $\bar{W}^{2,p}(\Lz M')$ and $\partial W^{2,p}(\Lz M')$ the closure and boundary of $W^{2,p}(\Lz M')$ in the Banach space $W^{2,p}(\operatorname{Sym}M')$ respectively. Note that some points at the boundary $\partial W^{2,p}(\Lz M')$ are not everywhere degenerate and hence may have positive $n$-dimensional Hausdorff measure, which is out of line with the intended properties of our variational framework. For this reason we introduce the \textit{essentially degenerate boundary} $\tilde{\partial}W^{2,p}(\Lz M')$ which consists of almost everywhere degenerate symmetric $(0,2)$-tensor fields which occur as limits of non-degenerate metrics in $W^{2,p}(\Lz M')$. It is clear that $\tilde{\partial}W^{2,p}(\Lz M')$ is a subset of $\partial W^{2,p}(\Lz M')$. Likewise, we define the \textit{essentially degenerate closure} $\tilde{W}^{2,p}(\Lz M') = W^{2,p}(\Lz M') \cup \tilde{\partial}W^{2,p}(\Lz M')$. This is also a cone that contains as a special case the zero section $0\in W^{2,p}(\operatorname{Sym}M')$ discussed previously, corresponding to the vertex of the cone and featuring the highest degree of degeneracy. Note that in view of the Morrey embedding applicable to the chosen regularity, ``almost everywhere'' in this case actually implies ``everywhere'', but the notion as we define it can be generalized for spaces of lower regularity. To generalize the discussion in Section \ref{Disconnected topological variations}, one may proceed to define variations which originate from any point at the essentially degenerate boundary rather than the vertex alone. The following is the direct analogue of Definition \ref{inductive limit geometric variations} given the above considerations.

\begin{definition}
\noindent
\begin{enumerate}
\item Let $(M,g;\Omega)\in \D_\EH$, $M'$ be a closed manifold and let $\mathcal{U}(M,g;\Omega)$ be a neighborhood of the zero section in $W^{2,p}_c(\operatorname{Sym}\Omega)$ such that the variational map $\Phi_{(M,g;\Omega,M')}: \mathcal{U}(M,g;\Omega) \times \tilde{W}^{2,p}(\Lz M') \rightarrow \D_\EH$,
\begin{equation}
\Phi_{(M,g;\Omega, M')}(h,g') = \begin{dcases}
	(M,g+h;\Omega) & \text{if } g' \in \tilde{\partial} W^{2,p}(\Lz M') \\
  	(M\sqcup M', (g+h) \oplus g' ; \Omega \sqcup M') & \text{if } g' \in W^{2,p}(\Lz M')
  \end{dcases}.
\end{equation}
is well defined.
\item We denote by $\tau_0$ the final topology on $\D_\EH$ generated by the family of maps as above.
\end{enumerate}
\end{definition}

Since the variational maps defining $\tau_1$ are restrictions of the variational maps defining $\tau_0$ along rays emerging from the origin, it is clear that $\tau_1$ is finer than $\tau_0$. The study of the boundary $\partial W^{2,p}(\Lz M')$ or essentially degenerate boundary $\tilde{\partial}W^{2,p}(\Lz M')$ can turn out to be very interesting in their own right, and topology $\tau_0$ certainly has the merit of generality, but it also has significant disadvantages for the purposes of our framework. Since a general point at the boundary is not as conveniently placed as the vertex of the cone, one setback is the absence of an easily comprehensible tangent structure at the boundary; if it exists, such structure would explicitly depend on the point on the boundary and the degree of degeneracy there. Since the ultimate goal is to introduce derivatives and study them in a more or less unified scheme, this already poses a significant challenge, without however being fatal by itself. What is fatal and makes this topology unsuitable for Einstein-Hilbert is the inevitable discontinuity which emerges as a consequence of Proposition \ref{Collapsing product}.

\begin{proposition}\label{tau0 discontinuity}
$S_\EH:(\D_\EH,\tau_0)\rightarrow \R$ is everywhere discontinuous for any $n\geq 2$ and any value of $p$.
\end{proposition}

\begin{proof}
Since $\tau_0 \subset \tau_1$, it follows by Theorem \ref{Disjoint continuity thm} that $S_\EH:(\D_\EH,\tau_0)\rightarrow \R$ is also everywhere discontinuous for $n = 2$. In the remaining case $n\geq 3$, by considering $S^1 \times N$ and $g_\epsilon$ as in Proposition \ref{Collapsing product} we have that $S_\EH\circ\Phi_{(M,g;\Omega,S^1\times N)}$ is discontinuous for any $(M,g;\Omega)\in \D_\EH$, which proves the claim.
\end{proof}

Note that it is possible to obtain a tangent cone structure at $0\in W^{2,p}(\operatorname{Sym}M')$ without restricting to rays. One may restrict the variational maps of $\tau_0$ to include only the vertex of the cone and no other boundary points, i.e. to consider infinitesimal disconnected components that are only growing out of a point geometrically. In that case the tangent structure at the vertex is precisely a closed tangent cone which coincides with the closure $\bar{W}^{2,p}(\Lz M')$; this fact holds for open cones in Banach spaces in general.

\begin{definition}
\noindent
\begin{enumerate}
\item For a closed manifold $M'$, we denote $\dot{W}^{2,p}(\Lz M') = W^{2,p}(\Lz M') \cup \{ 0 \}$.
\item Let $(M,g;\Omega)\in \D_\EH$, $M'$ be a closed manifold and let $\mathcal{U}(M,g;\Omega)$ be a neighborhood of the zero section in $W^{2,p}_c(\operatorname{Sym}\Omega)$ such that the variational map $\Phi_{(M,g;\Omega,M')}: \mathcal{U}(M,g;\Omega) \times \dot{W}^{2,p}(\Lz M') \rightarrow \D_\EH$,
\begin{equation}
\Phi_{(M,g;\Omega, M')}(h,g') = \begin{dcases}
	(M,g+h;\Omega) & \text{if } g' = 0 \\
  	(M\sqcup M', (g+h) \oplus g' ; \Omega \sqcup M') & \text{if } g' \in W^{2,p}(\Lz M')
  \end{dcases}.
\end{equation}
is well defined.
\item We denote by $\tau_0'$ the final topology on $\D_\EH$ generated by the family of maps as above.
\end{enumerate}
\end{definition}

Note that by the definition of final topology there holds $\tau_0 \subset \tau_0' \subset \tau_1$. Topology $\tau_0'$ thus constitutes a middle ground which still possesses a tangent cone structure in topological directions and is not as restrictive as in only permitting the vertexes to be approached along rays. Interestingly, even this is not enough to break the discontinuity spell. The culprit is the fact that the boundary $\partial W^{2,p}(\Lz M')$ also belongs to the tangent cone, i.e. it is possible to approach zero by non-degenerate metrics which still become tangentially degenerate in the limit. This possibility is ruled out in $\tau_1$, as its tangent cone along $M'$ is only the open cone $W^{2,p}(\Lz M')$ with the origin adjoined.

\begin{proposition}\label{tau0' discontinuity}
$S_\EH:(\D_\EH,\tau_0')\rightarrow \R$ is everywhere discontinuous for any $n\geq 2$ and any value of $p$.
\end{proposition}

\begin{proof}
The case $n=2$ is covered in Proposition \ref{tau0 discontinuity} and Theorem \ref{Disjoint continuity thm}. For $n\geq 3$, let $0< \delta < 2/(n-2)$ and consider the sequence of metrics $g_\epsilon' = \epsilon^\delta g_\epsilon$ in $S^1 \times N$ as in Proposition \ref{Blow-up}. Since the $g_\epsilon$ are uniformly bounded up to any order and $\delta>0$, it follows that $g_\epsilon' \rightarrow 0$ in $W^{2,p}(\operatorname{Sym} S^1\times N)$. However,
\begin{equation}
\int_{S^1\times N}  \mathcal{L}[g_\epsilon'] = O\big( \epsilon^{\frac{n-2}{2}\delta-1} \big),
\end{equation}
which is divergent since $\delta<2/(n-2)$. This makes $S_\EH \circ \Phi_{(M,g;\Omega,S^1\times N)}$ discontinuous for any $(M,g;\Omega) \in \D_\EH$, which proves the claim.
\end{proof}

This does not rule out the possibility for a topology coarser than $\tau_1$ that still possesses a tangent cone structure in topological directions and also preserves continuity of Einstein-Hilbert, at least within the limits of what Theorem \ref{Disjoint continuity thm} has specified as essential obstructions. Searching for such a topology, the above considerations make it clear that the boundary $\partial W^{2,p}(\Lz M')$ must be avoided as zero is being approached, even tangentially. This strongly points to some kind of uniform non-degeneracy condition. Indeed, looking at the scalar curvature formula \eqref{scalar curvature}, one can be easily convinced that pathologies arise due to the fact that the inverse $g^{\mu\nu}$ is uncontrolled in the Sobolev topology. This motivates the following.

\begin{definition}
\noindent
\begin{enumerate}
\item Let $M'$ be a closed manifold, $\hat{g}$ an auxiliary smooth Riemannian metric on $M'$, and let $0<\Lambda<\infty$ be a real number. The \textit{uniformly non-degenerate open cone} $\mathcal{C}(M',\hat{g},\Lambda)$ consists of all metrics $g'\in W^{2,p}(\Lz M')$ such that
\begin{equation}
\bigg\| \bigg( \frac{g'}{\| g' \|_{W^{2,p}(M',\hat{g})}} \bigg)^{-1} \bigg\|_{L^{\infty}(M',\hat{g})} < \Lambda.
\end{equation}
Denote by $\dot{\mathcal{C}}(M',\hat{g},\Lambda) = \mathcal{C}(M',\hat{g},\Lambda) \cup \{ 0 \}$.

\item Let $(M,g;\Omega)\in \D_\EH$, $M'$ be a closed manifold and let $\mathcal{U}(M,g;\Omega)$ be a neighborhood of the zero section in $W^{2,p}_c(\operatorname{Sym}\Omega)$ such that the variational map $\Phi_{(M,g;\Omega,M',\hat{g},\Lambda)}: \mathcal{U}(M,g;\Omega) \times \dot{\mathcal{C}}(M',\hat{g},\Lambda) \rightarrow \D_\EH$,
\begin{equation}
\Phi_{(M,g;\Omega, M',\hat{g},\Lambda)}(h,g') = \begin{dcases}
	(M,g+h;\Omega) & \text{if } g' = 0 \\
  	(M\sqcup M', (g+h) \oplus g' ; \Omega \sqcup M') & \text{if } g' \in \mathcal{C}(M',\hat{g},\Lambda)
  \end{dcases}.
\end{equation}
is well defined.
\item We denote by $\tau_\mathrm{UND}$ the final topology on $\D_\EH$ generated by the family of maps as above.
\end{enumerate}
\end{definition}

It is clear that $\tau_0' \subset \tau_\mathrm{UND} \subset \tau_1$, and that degenerate directions are excluded from the tangent cone. Indeed, the tangent cone of $\dot{\mathcal{C}}(M',\hat{g},\Lambda)$ at the vertex is the uniformly non-degenerate closed cone $\bar{\mathcal{C}}(M',\hat{g},\Lambda)$, which stays clear of the nontrivial boundary $\partial W^{2,p}(\Lz M') \setminus \{0\}$ for any value of the parameters. Moreover, we show bellow that continuity of Einstein-Hilbert is preserved in the sense of Theorem \ref{Disjoint continuity thm}.

\begin{theorem}
$(S_\EH,\tau_\mathrm{UND})$ is continuous for $n>2$ and $p>n/2$.
\end{theorem}

\begin{proof}
Let $g' \in \mathcal{C}(M',\hat{g},\Lambda)$. Then by straightforward calculation and use of the Sobolev and Morrey embeddings we obtain
\begin{equation}
\bigg| \int_{M'} \mathcal{L}\bigg[  \frac{g'}{\| g' \|_{W^{2,p}(M',\hat{g})}} \bigg]\  \bigg| \leq C(n,p,\hat{g},\Lambda).
\end{equation}
Then an application of Lemma \ref{Scaling Lemma} yields
\begin{equation}
\bigg| \int_{M'} \mathcal{L}[g']\  \bigg| \leq C(n,p,\hat{g},\Lambda) \| g' \|_{W^{2,p}(M',\hat{g})}^{\frac{n-2}{2}},
\end{equation}
which reduces the problem to the scaling asymptotics considered in Theorem \ref{Disjoint continuity thm}. It follows that for $n>2$, the map $S_\EH \circ \Phi_{(M,g;\Omega,M',\hat{g},\lambda,\Lambda)}$ is continuous for any $(M,g;\Omega) \in \D_\EH$. This completes the proof.
\end{proof}

The proof in addition shows that in the case the topology is made coarser by considering variational maps along uniformly non-degenerate cones rather than rays, continuity still rests on the same scaling arguments. In that sense, working with $\tau_1$ in this setting comes without loss of generality.

\section{The effect of quadratic curvature terms} \label{Quadratic terms}

In this section we briefly demonstrate how topological variations interact with quadratic curvature terms in the action, which also correspond to popular models of modified gravity; see Steele \cite{stelle1978} for some foundational physical motivation, also Stavrinos and Saridakis \cite{stavrinos2023} for a modern account on modified gravity theories. For our purposes, such models constitute an ideal example featuring the effect of extra terms in the regime of topological variations, most notably the displacement of critical dimension. The most general action of this form is 
\begin{align}\label{Quadratic action}
\begin{split}
S_\Q[M,g;\Omega] = \int_\Omega (-2\Lambda + R +\alpha R^2 + \beta R_{\mu\nu}R^{\mu\nu} + \gamma R_{\rho\sigma\mu\nu}R^{\rho\sigma\mu\nu})|\det g_{\mu\nu}\,|^{1/2}\,dx
\end{split}
\end{align}
where $\alpha,\beta$ and $\gamma$ are real constants, not all zero, and we have also included a cosmological constant term $\Lambda$, which we could have introduced in previous sections; we will see however that it is dominated by higher order terms due to scaling properties. The effect on critical dimension can be observed even when working with smooth metrics, although it is possible to conduct a thorough analysis in the weak Sobolev framework similar to our previous analysis of Einstein-Hilbert; here we opt for the simpler setup to avoid making the discussion repetitive.

Let us investigate how action \eqref{Quadratic action} changes under the infinitesimal disconnected topological variation \eqref{disjoint infinitesimal variation}. The difference in action is
\begin{align}\label{quadratic variation 1}
\begin{split}
\Delta S_\Q (\epsilon) &= \int_{M'} (-2\Lambda + R[\epsilon g'] + \alpha R^2[\epsilon g']  +\beta R_{\mu\nu}[\epsilon g']R^{\mu\nu}[\epsilon g'] \\
	&\hspace{0.5cm} +\gamma R_{\rho\sigma\mu\nu}[\epsilon g']R^{\rho\sigma\mu\nu}[\epsilon g'])|\det \epsilon g'_{\mu\nu}\,|^{1/2}\,dx'.
\end{split}
\end{align}
The first term remains constant, and we already know how the scalar curvature term scales. Since the curvature tensor ${R^\rho}_{\mu\sigma\nu}$ depends only on the Levi-Civita connection and the later is scale invariant, it follows that both the curvature tensor and the Ricci contraction $R_{\mu\nu} = {R^\rho}_{\mu\rho\nu}$ are scale invariant. So it follows that
\begin{align}
\begin{split}
R_{\mu\nu}[\epsilon g']R^{\mu\nu}[\epsilon g'] &= (\epsilon g')^{\mu\alpha}(\epsilon g')^{\nu\beta} R_{\mu\nu}[\epsilon g']R_{\alpha\beta}[\epsilon g'] \\
	&= \epsilon^{-2} R_{\mu\nu}[g']R^{\mu\nu}[g'],
\end{split}
\end{align}
and similarly
\begin{equation}
R_{\rho\sigma\mu\nu}[\epsilon g']R^{\rho\sigma\mu\nu}[\epsilon g'] = \epsilon^{-2}R_{\rho\sigma\mu\nu}[g']R^{\rho\sigma\mu\nu}[g'].
\end{equation}
Substituting these back to \eqref{quadratic variation 1} we obtain
\begin{align}\label{quadratic variation 2}
\begin{split}
\Delta S_\Q (\epsilon) &= \int_{M'} (-2\Lambda + \epsilon^{-1} R[g'] + \alpha \epsilon^{-2} R^2[g'] +\beta \epsilon^{-2} R_{\mu\nu}[g']R^{\mu\nu}[g'] \\
	&\hspace{0.5cm} +\gamma \epsilon^{-2} R_{\rho\sigma\mu\nu}[g']R^{\rho\sigma\mu\nu}[g'])\epsilon^{n/2}|\det g'_{\mu\nu}\,|^{1/2}\,dx',
\end{split}
\end{align}
while collecting powers of $\epsilon$ yields
\begin{align}\label{quadratic variation 3}
\begin{split}
\Delta S_\Q (\epsilon) &= \epsilon^{\frac{n-4}{2}}\int_{M'} (-2\epsilon^{2}\Lambda + \epsilon R[g'] + \alpha R^2[g'] +\beta R_{\mu\nu}[g']R^{\mu\nu}[g'] \\
	&\hspace{0.5cm} +\gamma R_{\rho\sigma\mu\nu}[g']R^{\rho\sigma\mu\nu}[g'])|\det g'_{\mu\nu}\,|^{1/2}\,dx'.
\end{split}
\end{align}
It follows that the continuity condition $\Delta S_\Q (\epsilon)\rightarrow 0$ as $\epsilon \rightarrow 0^+$ is generically satisfied only if $n>4$, not $n>2$ as in Einstein-Hilbert. So, unless $n>4$, quadratic actions are not even continuous under topological variations. As for the derivative, we see that the difference quotient is
\begin{align}\label{quadratic variation 4}
\begin{split}
\frac{\Delta S_\Q (\epsilon)}{\epsilon} &= \epsilon^{\frac{n-6}{2}}\int_{M'} (-2\epsilon^{2}\Lambda + \epsilon R[g'] + \alpha R^2[g'] + \beta R_{\mu\nu}[g']R^{\mu\nu}[g'] \\
	&\hspace{0.5cm} +\gamma R_{\rho\sigma\mu\nu}[g']R^{\rho\sigma\mu\nu}[g'])|\det g'_{\mu\nu}\,|^{1/2}\,dx'.
\end{split}
\end{align}
If the limit $\epsilon\rightarrow 0^+$ is to be well defined for arbitrary $g'$, there should hold $n\geq 6$. If $n=6$, the derivative is
\begin{align}\label{quadratic variation 5}
\begin{split}
\delta_{(M,g;\Omega)}S_\Q[M',g'] = \int_{M'} (\alpha R^2[g'] +\beta R_{\mu\nu}[g']R^{\mu\nu}[g'] + \gamma R_{\rho\sigma\mu\nu}[g']R^{\rho\sigma\mu\nu}[g'])|\det g'_{\mu\nu}\,|^{1/2}\,dx',
\end{split}
\end{align}
in which case we see that only quadratic terms survive; for the same reason, the cosmological constant term would not survive in the topological functional derivative of the Einstein-Hilbert action when $n=4$. For $n>6$, the derivative is identically zero, and every geometrically stationary point, satisfying the Euler-Lagrange equations of \eqref{Quadratic action}, will also be stationary in the extended framework including disconnected topological variations. The connected version of this is analogous. 

\begin{remark}[The Einstein-Gauss-Bonnet model]
A considerable amount of attention has been given to the Gauss-Bonnet term, which is also the subject of Tsilioukas et al. \cite{tsilioukas2024}, where a heuristic attempt to define a topological derivative is made; for more background, see references therein. The corresponding action is the Einstein-Gauss-Bonnet action
\begin{align}\label{EGB action}
\begin{split}
S_\mathrm{EGB}[M,g;\Omega] = \int_\Omega R|\det g_{\mu\nu}\,|^{1/2}\,dx + \alpha \int_\Omega (R^2 -4 R_{\mu\nu}R^{\mu\nu} +R_{\mu\nu\rho\sigma}R^{\mu\nu\rho\sigma})|\det g_{\mu\nu}\,|^{1/2}\,dx,
\end{split}
\end{align}
where $\alpha$ is a non-zero constant. In the regime of topological variations, this is interesting due to the Chern-Gauss-Bonnet theorem \cite{chern1944}, which states that in dimension $n=4$ there holds
\begin{align}
\begin{split}
\int_\Omega (R^2 -4 R_{\mu\nu}R^{\mu\nu} +R_{\mu\nu\rho\sigma}R^{\mu\nu\rho\sigma})|\det g_{\mu\nu}\,|^{1/2}\,dx = \frac{1}{32\pi^2}\chi(\Omega) +\int_{\partial \Omega} Q,
\end{split}
\end{align}
where $\chi(\Omega)$ is the Euler characteristic of $\bar{\Omega}$ as a compact manifold with boundary and $Q$ is a $3$-form on the boundary of $\Omega$. The change of \eqref{EGB action} under the infinitesimal disconnected topological variation \eqref{disjoint infinitesimal variation} then amounts to
\begin{equation}\label{EGB change}
\Delta S_\mathrm{EGB}(\epsilon) = \epsilon \int_{M'} R[g']|\det g'_{\mu\nu}\,|^{1/2}\,dx' + \frac{\alpha}{32\pi^2}\chi(M').
\end{equation}
It follows that $\Delta S_\mathrm{EGB}(\epsilon) \rightarrow \alpha\chi(M')/32\pi^2$ as $\epsilon\rightarrow 0^+$, which renders the action discontinuous under topological variations that have $\chi(M') \neq 0$. This already shows that a derivative in this case is ill-defined for topological variations. But even in the case where $\chi(M')=0$, the problem of the action not admitting critical points in dimension $n=4$ still persists, as with plain Einstein-Hilbert.
\end{remark}

\end{document}